\documentclass[letterpaper, 10pt, conference]{ieeeconf}
\IEEEoverridecommandlockouts 
\usepackage{amsmath,amssymb,url}
\usepackage{graphicx,subfigure,hyperref}
\usepackage{color}

\newcommand{\norm}[1]{\ensuremath{\left\| #1 \right\|}}

\newcommand{\bracket}[1]{\ensuremath{\left[ #1 \right]}}
\newcommand{\braces}[1]{\ensuremath{\left\{ #1 \right\}}}
\newcommand{\parenth}[1]{\ensuremath{\left( #1 \right)}}

\newcommand{\refeqn}[1]{(\ref{eqn:#1})}
\newcommand{\reffig}[1]{Fig. \ref{fig:#1}}

\newcommand{\trs}[1]{\mbox{tr}\ensuremath{\!\bracket{#1}}}

\newcommand{\SO}{\ensuremath{\mathsf{SO(3)}}}
\newcommand{\T}{\ensuremath{\mathsf{T}}}

\newcommand{\so}{\ensuremath{\mathfrak{so}(3)}}

\renewcommand{\Re}{\ensuremath{\mathbb{R}}}
\newcommand{\Sph}{\ensuremath{\mathsf{S}}}

\title{\LARGE \bf
Geometric Adaptive Control for Aerial Transportation of a Rigid Body}

\author{Taeyoung Lee
\thanks{Taeyoung Lee, Mechanical and Aerospace Engineering, George Washington University, Washington, DC 20052 {\tt tylee@gwu.edu}}%
\thanks{This research has been supported in part by NSF under the grants CMMI-1243000 (transferred from 1029551), CMMI-1335008, and CNS-1337722.}
}

\newtheorem{prop}{Proposition}

\begin{document}
\allowdisplaybreaks

\maketitle \thispagestyle{empty} \pagestyle{empty}

\begin{abstract}
This paper is focused on tracking control for a rigid body payload, that is connected to an arbitrary number of quadrotor unmanned aerial vehicles via rigid links. A geometric adaptive controller is constructed such that the payload asymptotically follows a given desired trajectory for its position and attitude in the presence of uncertainties. The coupled dynamics between the rigid body payload, links, and quadrotors are explicitly incorporated into control system design and stability analysis. These are developed directly on the nonlinear configuration manifold in a coordinate-free fashion to avoid singularities and complexities that are associated with local parameterizations. 
\end{abstract}

\section{Introduction}




By utilizing the high thrust-to-weight ratio, quadrotor unmanned aerial vehicles have been envisaged for aerial load transportation~\cite{PalCruIRAM12,MicFinAR11,MazKonJIRS10}. Most of the existing results for the control of quadrotors to transport a cable-suspended payload are based on the assumption that the dynamics of the payload is decoupled from the dynamics of quadrotors. For example, the effects of the payload are considered as arbitrary external force and torque exerted to quadrotors~\cite{MicFinAR11}. As such, these results may not be suitable for agile load transportation where the motion of cable and payload should be actively suppressed.

Recently, the full dynamic model for an arbitrary number of quadrotors transporting a payload are developed, and based on that, geometric tracking controllers are constructed in an intrinsic fashion. In particular, autonomous transportation of a point mass connected to quadrotors via rigid links is developed in~\cite{LeeSrePICDC13}. It has been generalized into a more realistic dynamic model that considers the deformation of cables in~\cite{GooLeePACC14}, and also the attitude dynamics of a payload, that is considered as a rigid body instead of a point mass, is incorporated in~\cite{LeePICDC14}. However, these results are based on the assumption that the exact properties of the quadrotors and the payload are available, and that there are no external disturbances, thereby making it challenging to implement those results in  actual hardware systems. 

The objective of this paper is to construct a control system for an arbitrary number of quadrotors connected to a rigid body payload via rigid links with explicit consideration on uncertainties. A coordinate-free form of the equations of motion that have been developed in~\cite{LeePICDC14} is extended to include the effects of unknown, but fixed forces and moments acting on each of the quadrotors, the cables, and the payload. A geometric nonlinear adaptive control system is designed such that both the position and the attitude of the payload asymptotically follow their desired trajectories, while maintaining a certain formation of quadrotors relative to the payload. 

The unique property is that the coupled dynamics of the payload, the cables, and quadrotors are explicitly incorporated in control system design for agile load transportations where the motion of the payload relative to the quadrotors are excited nontrivially. Another distinct feature is that the equations of motion and the control systems are developed directly on the nonlinear configuration manifold intrinsically. Therefore, singularities of local parameterization are completely avoided.

As such, the proposed control system is particularly useful for rapid and safe payload transportation in complex terrain, where the position and attitude of the payload should be controlled concurrently. Most of the existing control systems of aerial load transportation suffer from limited agility as they are based on reactive assumptions that ignore the inherent complexities in the dynamics of aerial load transportation. The proposed control system explicitly integrates the comprehensive dynamic characteristics to achieve extreme maneuverability in aerial load transportation. To the author's best knowledge, nonlinear adaptive tracking controls of a cable-suspended rigid body with uncertainties have not been studied as mathematically rigorously as presented in this paper.

\section{Problem Formulation}\label{sec:DM}

Consider $n$ quadrotor UAVs that are connected to a payload, that is modeled as a rigid body, via massless links (see Figure \ref{fig:DM}). Throughout this paper, the variables related to the payload is denoted by the subscript $0$, and the variables for the $i$-th quadrotor are denoted by the subscript $i$, which is assumed to be an element of $\mathcal{I}=\{1,\cdots\, n\}$ if not specified. We choose an inertial reference frame $\{\vec e_1,\vec e_2,\vec e_3\}$ and body-fixed frames $\{\vec b_{j_1},\vec b_{j_2},\vec b_{j_3}\}$ for $0\leq j\leq n$ as follows. For the inertial frame, the third axis $\vec e_3$ points downward along the gravity and the other axes are chosen to form an orthonormal frame. 

\begin{figure}
\centerline{\setlength{\unitlength}{0.1\columnwidth}\footnotesize
\begin{picture}(10,6.5)(0.0,0)
\put(1.25,0){\includegraphics[width=0.75\columnwidth]{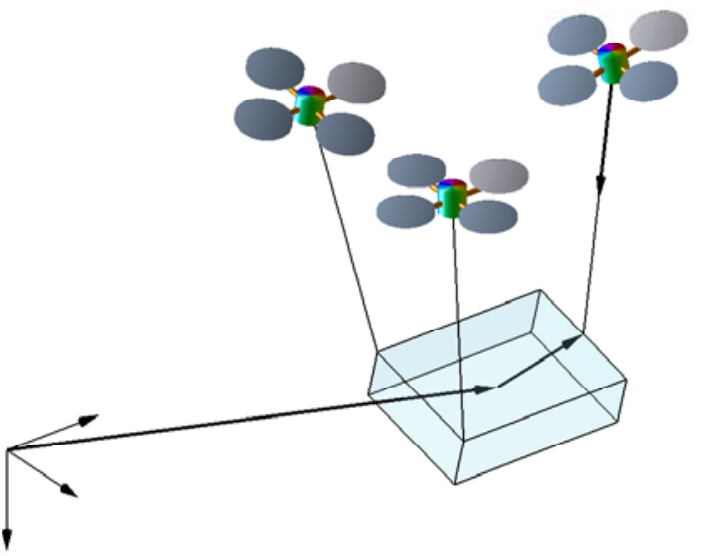}}
\put(2.35,1.45){$\vec e_1$}
\put(1.9,0.3){$\vec e_2$}
\put(1.4,-0.1){$\vec e_3$}
\put(3.1,1.0){$x_0\in\Re^3$}
\put(6.7,0.5){\shortstack[c]{$m_0,J_0$\\$R_0\in\SO$}}
\put(5.3,5.3){\shortstack[c]{$m_i,J_i$\\$R_i\in\SO$}}
\put(7.8,4.2){$q_i\in\Sph^2$}
\put(6.8,1.8){$\rho_i$}
\put(7.65,3.3){$l_i$}
\end{picture}}
\caption{Dynamics model: $n$ quadrotors are connect to a rigid body $m_0$ via massless links $l_i$. The configuration manifold is $\Re^3\times\SO\times(\Sph^2\times\SO)^n$.}\label{fig:DM}
\vspace*{-0.1cm}
\end{figure}
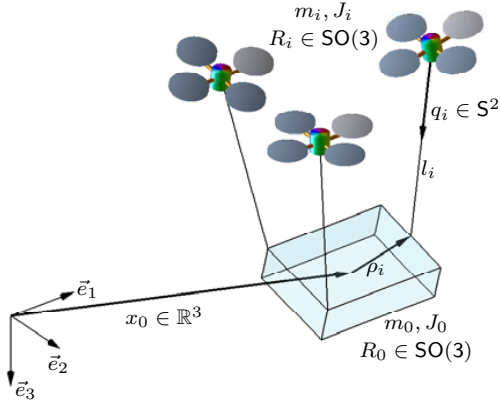

The location of the mass center of the payload is denoted by $x_0\in\Re^3$, and its attitude is given by $R_0\in\SO$, where the special orthogonal group is defined by $\SO=\{R\in\Re^{3\times 3}\,|\, R^TR=I,\,\mathrm{det}[R]=1\}$. Let $\rho_i\in\Re^3$ be the point on the payload where the $i$-th link is attached, and it is represented with respect to the zeroth body-fixed frame. The other end of the link is attached to the mass center of the $i$-th quadrotor. The direction of the link from the mass center of the $i$-th quadrotor toward the payload is defined by the unit-vector $q_i\in\Sph^2$, where $\Sph^2=\{q\in\Re^3\,|\,\|q\|=1\}$, and the length of the $i$-th link is denoted by $l_i\in\Re$.

Let $x_i\in\Re^3$ be the location of the mass center of the $i$-th quadrotor with respect to the inertial frame. As the link is assumed to be rigid, we have $x_i=x_0+R_0\rho_i-l_i q_i$. The attitude of the $i$-th quadrotor is defined by $R_i\in\SO$, which represents the linear transformation of the representation of a vector from the $i$-th body-fixed frame to the inertial frame. 

In summary, the configuration of the presented system is described by the position $x_0$ and the attitude $R_0$ of the payload, the direction $q_i$ of the links, and the attitudes $R_i$ of the quadrotors. The corresponding configuration manifold of this system is $\mathsf{Q}=\Re^3\times\SO\times(\Sph^2\times \SO)^n$.

The mass and the inertia matrix of the payload are denoted by $m_0\in\Re$ and $J_0\in\Re^{3\times 3}$, respectively. The dynamic model of each quadrotor is identical to~\cite{LeeLeoPICDC10}. The mass and the inertia matrix of the $i$-th quadrotor are denoted by $m_i\in\Re$ and $J_i\in\Re^{3\times 3}$, respectively. The $i$-th quadrotor can generates a thrust $-f_iR_ie_3\in\Re^3$ with respect to the inertial frame, where $f_i\in\Re$ is the total thrust magnitude and $e_3=[0,0,1]^T\in\Re^3$. It also generates a moment $M_i\in\Re^3$ with respect to its body-fixed frame. The control input of this system corresponds to $\{f_i,M_i\}_{1\leq i\leq n}$.

In this paper, the external disturbances are modeled as follows. The disturbance force  and moment  acting on the payload, namely $\Delta_{x_0},\Delta_{R_0}\in\Re^3$ are expressed as
\begin{align}
\Delta_{x_0}=\Phi_{x_0}(t,\mathsf{q},\dot{\mathsf{q}})\theta_{x_0},\quad
\Delta_{R_0}=\Phi_{R_0}(t,\mathsf{q},\dot{\mathsf{q}})\theta_{R_0},
\end{align}
where $\Phi_{x_0},\Phi_{R_0}:\Re\times\mathsf{TQ}\rightarrow \Re^{3\times n_\theta}$ denote matrix-valued, known function of the time $t$ and the tangent vector $(\mathsf{q},\dot{\mathsf{q}})\in\T_{\mathsf{q}}\mathsf{Q}$ of the configuration manifold, i.e., $q=(x_0,R_0,q_1,\ldots,q_n,R_1,\ldots R_n)$, and $\theta_{x_0},\theta_{R_0}\in\Re^{n_\theta\times 1}$ are fixed, unknown parameters for some $n_\theta$. This type of uncertainties are popular in the literature of adaptive controls, and they may represent various modeling errors or disturbances, such as the uncertainties in the mass and the inertia matrix of the payload. Similarly, the disturbance force and moment acting on the $i$-th quadrotors are given by
\begin{align}
\Delta_{x_i}=\Phi_{x_i}(t,\mathsf{q},\dot{\mathsf{q}})\theta_{x_i},\quad
\Delta_{R_i}=\Phi_{R_i}(t,\mathsf{q},\dot{\mathsf{q}})\theta_{R_i},
\end{align}
where $\Phi_{x_i},\Phi_{R_i}:\Re\times\mathsf{TQ}\rightarrow \Re^{3\times n_\theta}$ and $\theta_{x_i},\theta_{R_i}\in\Re^{n_\theta\times 1}$. Here, the disturbance forces are represented with respect to the inertial frame, and the disturbance moments are represented with respect to the corresponding body-fixed frame. 

Throughout this paper, the 2-norm of a matrix $A$ is denoted by $\|A\|$, and its maximum eigenvalue and minimum eigenvalues are denoted by $\lambda_{M}[A]$ and $\lambda_{m}[A]$, respectively. The standard dot product is denoted by $x \cdot y = x^Ty$ for any $x,y\in\Re^3$.

\subsection{Equations of Motion}

The kinematic equations for the payload, quadrotors, and links are given by
\begin{gather}
\dot q_{i} = \omega_i\times q_i=\hat\omega_i q_i,\label{eqn:dotqi}\\
\dot R_0  = R_0\hat\Omega_0,\quad \dot R_i  = R_i\hat\Omega_i,\label{eqn:dotRi}
\end{gather}
where $\omega_i\in\Re^3$ is the angular velocity of the $i$-th link, satisfying $q_i\cdot\omega_i=0$, and $\Omega_0$ and $\Omega_i\in\Re^3$ are the angular velocities of the payload and the $i$-th quadrotor expressed with respect to its body-fixed frame, respectively. The \textit{hat map} $\hat\cdot:\Re^3\rightarrow\so$ is defined by the condition that $\hat x y=x\times y$ for all $x,y\in\Re^3$, and the inverse of the hat map is denoted by the \textit{vee map} $\vee:\so\rightarrow\Re^3$, where $\so$ denotes the set of $3\times 3$ skew-symmetric matrices, i.e., $\so=\{S\in\Re^{3\times 3}\,|\, S^T=-S\}$, and it corresponds to the Lie algebra of $\SO$. 

We derive equations of motion according to Lagrangian mechanics. The velocity of the $i$-th quadrotor is given by $\dot x_i = \dot x_0+\dot R_0\rho_i - l_i\dot q_i$. The kinetic energy of the system is composed of the translational kinetic energy and the rotational kinetic energy of the payload and quadrotors:
\begin{align}
\mathcal{T} & = \frac{1}{2}m_0 \|\dot x_0\|^2 + \frac{1}{2}\Omega_0\cdot J_0\Omega_0\nonumber\\
&\quad +\sum_{i=1}^n \frac{1}{2} m_i \|\dot x_0+\dot R_0\rho_i - l_i\dot q_i\|^2+ \frac{1}{2}\Omega_i\cdot J_i\Omega_i.\label{eqn:TT}
\end{align}
The gravitational potential energy is given by
\begin{align}
\mathcal{U} = - m_0g e_3 \cdot x_0 - \sum_{i=1}^n m_ige_3\cdot (x_0+R_0\rho_i-l_iq_i),\label{eqn:UU}
\end{align}
where the unit-vector $e_3$ points downward along the gravitational acceleration as shown at \reffig{DM}. The resulting Lagrangian of the system is $\mathcal{L}=\mathcal{T}-\mathcal{U}$. 

The corresponding Euler-Lagrange equations have been developed according to Hamilton's principle in~\cite{LeePICDC14},  Here, it is generalized to include the effects of disturbances via the Lagrange-d'Alembert principle. Let the action integral be $\mathfrak{G}=\int_{t_0}^{t_f} \mathcal{L}\, dt$. Next, let the total control thrust at the $i$-th quadrotor with respect to the inertial frame be denoted by $u_i = -f_iR_ie_3\in\Re^3$ and the total control moment at the $i$-th quadrotor is defined as $M_i\in\Re^3$. There exist the disturbances $\Delta_{x_0},\Delta_{R_0}$ for the payload, and the disturbances $\Delta_{x_i},\Delta_{R_i}$ for the $i$-th quadrotor. The virtual work can be written as
\begin{align*}
\delta\mathcal{W} & = \int_{t_0}^{t_f} \sum_{i=1}^n (u_i+\Delta_{x_i})\cdot\braces{\delta x_0 + R_0\hat\eta_0\rho_i -l_i\xi_i\times q_i}\\
&\quad + \sum_{i=1}^n(M_i+\Delta_{R_i}) \cdot\eta_i+\Delta_{x_0}\cdot\delta x_0 + \Delta_{R_0}\cdot \eta_0\,dt.
\end{align*}
The Lagrange-d'Alembert principle states that $\delta\mathfrak{G}=-\delta\mathcal{W}$ for any variation of trajectories with fixed end points. This yields the following equations of motion (see~\cite{LeeACC151ext} for detailed derivations),
\begin{gather}
M_q (\ddot x_0-ge_3) - \sum_{i=1}^n m_iq_iq_i^T R_0\hat\rho_i\dot\Omega_0 = \Delta_{x_0}\nonumber\\
 +\sum_{i=1}^n u_i^\parallel+\Delta_{x_i}^\parallel-m_il_i\|\omega_i\|^2q_i- m_iq_iq_i^T R_0\hat\Omega_0^2\rho_i,
\label{eqn:ddotx0}\\
  (J_0-\sum_{i=1}^n m_i\hat\rho_i R_0^T q_iq_i^T R_0\hat\rho_i)\dot\Omega_0 \nonumber\\
 + \sum_{i=1}^n m_i\hat\rho_i R_0^Tq_iq_i^T(\ddot x_0-ge_3) + \hat\Omega_0   J_0\Omega_0
=\Delta_{R_0}\nonumber \\
+\sum_{i=1}^n \hat\rho_i R_0^T (u_i^\parallel+\Delta_{x_i}^\parallel-m_il_i\|\omega_i\|^2 q_i-m_iq_iq_i^TR_0\hat\Omega_0^2\rho_i),\label{eqn:dotW0}\\
\dot\omega_i =\frac{1}{l_i}\hat q_i(\ddot x_0-ge_3 - R_0\hat\rho_i\dot\Omega_0+
R_0\hat\Omega_0^2\rho_i)\nonumber\\
 -\frac{1}{m_il_i}\hat q_i(u_i^\perp+\Delta_{x_i}^\perp),\label{eqn:dotwi}\\
J_i\dot\Omega_i + \Omega_i\times J_i\Omega_i = M_i+\Delta_{R_i},\label{eqn:dotWi}
\end{gather}
where $M_q=m_y I + \sum_{i=1}^n m_i q_i q_i^T\in\Re^{3\times 3}$, which is symmetric, positive-definite for any $q_i$. 

Recall the vector $u_i\in\Re^3$ represents the control force at the $i$-th quadrotor, i.e., $u_i=-f_i R_i e_3$.  The vectors $u_i^\parallel$ and $u_i^\perp\in\Re^3$ denote the orthogonal projection of $u_i$ along $q_i$, and the orthogonal projection of $u_i$ to the plane  normal to $q_i$, respectively, i.e.,
\begin{align}
u_i^\parallel & = q_iq_i^T u_i,\\
u_i^\perp &= -\hat q_i^2 u_i =(I-q_iq_i^T) u_i.
\end{align}
Therefore, $u_i = u_i^\parallel + u_i^\perp$. Throughout this paper, the subscripts $\parallel$ and $\perp$ of a vector denote the component of the vector that is parallel to $q_i$ and the other component of the vector that is perpendicular to $q_i$. Similarly, the disturbance force at the $i$-th quadrotor is decomposed as
\begin{align}
\Delta_{x_i}^\parallel & = q_iq_i^T \Phi_{x_i}\theta_{x_i} \triangleq \Phi_{x_i}^\parallel \theta_{x_i},\label{eqn:delxiparallel}\\
\Delta_{x_i}^\perp & = (I-q_iq_i^T) \Phi_{x_i}\theta_{x_i} \triangleq \Phi_{x_i}^\perp \theta_{x_i}.\label{eqn:delxiperp}
\end{align}

\subsection{Tracking Problem}

Define a fixed matrix $\mathcal{P} \in \Re^{6\times 3n}$ as 
\begin{align}
\mathcal{P} = \begin{bmatrix} I_{3\times 3} & \cdots & I_{3\times 3} \\ \hat\rho_1 & \cdots & \hat\rho_n \end{bmatrix}. \label{eqn:P}
\end{align}
Recall that $\rho_i$ describe the point on the payload where the $i$-th link is attached. 
Assume the links are attached to the payload such that
\begin{align}
\mathrm{rank}[\mathcal{P}] \geq 6.\label{eqn:A1}
\end{align}
This is to guarantee that there exist enough degrees of freedom in control inputs for both the translational motion and the rotational maneuver of the payload. The assumption \refeqn{A1} requires that the number of quadrotor is at least three, i.e., $n\geq 3$.

It is also assumed that the bounds of the disturbance forces and moments are available, i.e., for known positive constant $B_\Phi,B_{\theta}\in\Re$, we have
\begin{align}
\max\{&\|\Phi_{x_{0}}\|,
\|\Phi_{R_{0}}\|,
\|[\Phi_{x_1}\|,\ldots,\|\Phi_{x_n}\|,\nonumber\\
&\|\Phi_{R_0}\|,\ldots,\|\Phi_{R_n}\|\} < B_{\Phi},\label{eqn:BPhi}\\
\max\{&\|\theta_{x_{0}}\|,
\|\theta_{R_{0}}\|,
\|[\theta_{x_1}\|,\ldots,\|\theta_{x_n}\|,\nonumber\\
&\|\theta_{R_0}\|,\ldots,\|\theta_{R_n}\|\} < B_{\theta}.\label{eqn:Btheta}
\end{align}

Suppose that the desired trajectories for the position and the attitude of the payload are given as smooth functions of time, namely $x_{0_d}(t)\in\Re^3$ and $R_{0_d}(t)\in\SO$. From the attitude kinematics equation, we have
\begin{align*}
\dot R_{0_d}(t) = R_{0_d}(t) \hat\Omega_{0_d}(t),
\end{align*}
where $\Omega_{0_d}(t)\in\Re^3$ corresponds to the desired angular velocity of the payload. It is assumed that the velocity and the acceleration of the desired trajectories are bounded by known constants.

We wish to design a control input of each quadrotor $\{f_i,M_i\}_{1\leq i\leq n}$ such that  the tracking errors asymptotically converge to zero along the solution of the controlled dynamics.

\section{Control System Design For Simplified Dynamic Model}\label{sec:SDM}

In this section, we consider a simplified dynamic model where the attitude dynamics of each quadrotor is ignored, and we design a control input by assuming that the thrust at each quadrotor, namely $u_i$ can be arbitrarily chosen. It corresponds to the case where every quadrotor is replaced by a fully actuated aerial vehicle that can generates a thrust along any direction arbitrarily. The effects of the attitude dynamics of quadrotors will be incorporated in the next section. 

In the simplified dynamic model given by \refeqn{ddotx0}-\refeqn{dotwi}, the dynamics of the payload are affected by the parallel components $u_i^\parallel$ of the thrusts, and the dynamics of the links are directly affected by the normal components $u_i^\perp$ of the thrusts. This structure motivates the following control system design procedure: first, the parallel components $u_i^\parallel$ are chosen such that the payload follows the desired position and attitude trajectory while yielding the desired direction of each link, namely $q_{i_d}\in\Sph^2$; next, the normal components $u_i^\perp$ are designed such that the actual direction of the links $q_i$ follows the desired direction $q_{i_d}$.

\subsection{Design of Parallel Components}

Let $a_i\in\Re^3$ be the acceleration of the point on the payload where the $i$-th link is attached, that is measured relative to the gravitational acceleration:
\begin{align}
a_i = \ddot x_0 - ge_3 + R_0\hat\Omega_0^2\rho_i -R_0\hat\rho_i\dot\Omega_0.\label{eqn:ai0}
\end{align}
The parallel component of the control input is chosen as
\begin{align}
u_i^\parallel & = \mu_i + m_il_i\|\omega_i\|^2q_i + m_i q_iq_i^T a_i,\label{eqn:uip}
\end{align}
where $\mu_i\in\Re^3$ is a virtual control input that is designed later, with a constraint that $\mu_i$ is parallel to $q_i$. Note that the expression of $u_i^\parallel$ is guaranteed to be parallel to $q_i$ due to the projection operator $q_iq_i^T$ at the last term of the right-hand side of the above expression.

The motivation for the proposed parallel components becomes clear if \refeqn{uip} is substituted into \refeqn{ddotx0}-\refeqn{dotW0} and rearranged to obtain
\begin{gather}
m_0 (\ddot x_0 -g e_3) = \Delta_{x_0}+\sum_{i=1}^n (\mu_i+\Delta_{x_i}^\parallel),\label{eqn:ddotx0s}\\
J_0\dot\Omega_0 +\hat\Omega_0 J_0\Omega_0 = \Delta_{R_0}+\sum_{i=1}^n \hat\rho_i R_0^T (\mu_i+\Delta_{x_i}^\parallel).\label{eqn:dotW0s}
\end{gather}
Therefore, considering a free-body diagram of the payload, the virtual control input $\mu_i$ corresponds to the force exerted to the payload by the $i$-link, or the tension of the $i$-th link in the absence of disturbances. 


Next, we determine the virtual control input $\mu_i$. As in~\cite{GooLeePECC13}, define position, attitude, and angular velocity tracking error vectors $e_{x_0},e_{R_0},e_{\Omega_0}\in\Re^3$ for the payload as
\begin{align*}
e_{x_0} & = x_0 -x_{0_d},\\
e_{R_0} & = \frac{1}{2} (R_{0_d}^T R_0-R_0^T R_{0_d})^\vee,\\
e_{\Omega_0} & = \Omega_0 - R_0^T R_{0_d}\Omega_{0_d}.
\end{align*}
The desired resultant control force $F_d\in\Re^3$ and moment $M_d\in\Re^3$ acting on the payload are given as
\begin{align}
F_d & = m_0 (-k_{x_0}e_{x_0} - k_{\dot x_0} \dot e_{x_0} + \ddot x_{0_d} - g e_3)\nonumber\\
&\quad -\Phi_{x_0}\bar\theta_{x_0} - \sum_{i=1}^n\Phi_{x_i}^\parallel\bar\theta_{x_i},\label{eqn:Fd}\\
M_d & = -k_{R_0} e_{R_0} - k_{\Omega_0} e_{\Omega_0}   +(R_0^TR_{0_d}\Omega_{0_d})^\wedge J_0 R_0^TR_{0_d} \Omega_{0_d}\nonumber\\
& \quad + J_0 R_0^T R_{0_d} \dot\Omega_{0_d}-\Phi_{R_0}\bar\theta_{R_0}-\sum_{i=1}^n\hat\rho_iR_0 \Phi_{x_i}^\parallel\bar\theta_{x_i},\label{eqn:Md}
\end{align}
for positive constants $k_{x_0},k_{\dot x_0},k_{R_0},k_{\Omega_0}\in\Re$. Here, the estimates of the unknown parameters $\theta_{x_0},\theta_{x_i},\theta_{R_0}$ are denoted by $\bar\theta_{x_0},\bar\theta_{x_i},\bar\theta_{R_0}\in\Re^{n_\theta}$. Adaptive control laws to update the estimates of disturbances are introduced later at Section \ref{sec:AL}.

These are the ideal resultant force and moment to achieve the control objectives. One may try to choose the virtual control input $\mu_i$ by making the expressions in the right-hand sides of \refeqn{ddotx0s} and \refeqn{dotW0s}, namely $\sum_{i}\mu_i$ and $\sum_i\hat\rho_iR_0^T\mu_i$, become identical to $F_d$ and $M_d$, respectively. But, this is not valid in general, as each $\mu_i$ is constrained to be parallel to $q_i$. Instead, we choose the desired value of $\mu_i$, without any constraint, such that 
\begin{align}
\sum_{i=1}^n \mu_{i_d} = F_d,\quad \sum_{i=1}^n \hat\rho_i R_0^T \mu_{i_d} = M_d,\label{eqn:muid0}
\end{align}
or equivalently, using the matrix $\mathcal{P}$ defined at \refeqn{P},
\begin{align*}
\mathcal{P}
\begin{bmatrix}
R_0^T\mu_{1_d} \\ \vdots \\ R_0^T\mu_{n_d}
\end{bmatrix}
=
\begin{bmatrix}
R_0^T F_d \\ M_d
\end{bmatrix}.
\end{align*}
From the assumption stated at \refeqn{A1}, there exists at least one solution to the above matrix equation for any $F_d, M_d$. Here, we find the minimum-norm solution given by
\begin{align}
\begin{bmatrix} \mu_{1_d}\\\vdots\\\mu_{n_d} \end{bmatrix} 
= \mathrm{diag}[R_0,\cdots R_0]\;\mathcal{P}^T (\mathcal{P}\mathcal{P}^T)^{-1}\begin{bmatrix} R_0^T F_d\\M_d\end{bmatrix}.\label{eqn:muid}
\end{align}
The virtual control input $\mu_i$ is selected as the projection of its desired value $\mu_{i_d}$ along $q_i$,
\begin{align}
\mu_i = (\mu_{i_d}\cdot q_i) q_i=q_iq_i^T\mu_{i_d},\label{eqn:mui}
\end{align}
and the desired direction of each link, namely $q_{i_d}\in\Sph^2$ is defined as
\begin{align}
q_{i_d} = -\frac{\mu_{i_d}}{\|\mu_{i_d}\|}.\label{eqn:qid}
\end{align}
It is straightforward to verify that when $q_{i}=q_{i_d}$, the resultant force and moment acting on the payload become identical to their desired values.


\subsection{Design of Normal Components}

Substituting \refeqn{ai0} into \refeqn{dotwi} and using \refeqn{delxiperp}, the equation of motion for the $i$-link is given by
\begin{align}
\dot\omega_i & = \frac{1}{l_i}\hat q_i a_i-\frac{1}{m_il_i}\hat q_i (u_i^\perp+\Delta_{x_i}^\perp).\label{eqn:widotf0}
\end{align}
Here, the normal component of the control input $u_i^\perp$ is chosen such that $q_i\rightarrow q_{i_d}$ as $t\rightarrow\infty$. Control systems for the unit-vectors on the two-sphere have been studied in~\cite{BulLew05,Wu12}. In this paper, we adopt the control system developed in terms of the angular velocity in~\cite{Wu12}, and we augment it with an adaptive control term to handle the disturbance $\Delta_{x_i}^\perp$.

For the given desired direction of each link, its desired angular velocity is obtained from the kinematics equation as
\begin{align*}
\omega_{i_d} = q_{i_d}\times \dot q_{i_d}.
\end{align*}
Define the direction and the angular velocity tracking error vectors for the $i$-th link, namely $e_{q_i},e_{\omega_i}\in\Re^3$ as 
\begin{align*}
e_{q_i} & = q_{i_d}\times q_i,\\
e_{\omega_i} & = \omega_i + \hat q_i^2\omega_{i_d}.
\end{align*}
For positive constants $k_{q},k_{\omega}\in\Re$, the normal component of the control input is chosen as
\begin{align}
u_i^\perp & = m_il_i\hat q_i \{-k_q e_{q_i} -k_{\omega}e_{\omega_i} -(q_i\cdot\omega_{i_d})\dot q_i -\hat q_i^2\dot\omega_d\}\nonumber\\
&\quad - m_i\hat q_i^2 a_i- \Phi_{x_i}^\perp\bar\theta_{x_i}.\label{eqn:uiperp}
\end{align}

Note that the expression of $u_i^\perp$ is perpendicular to $q_i$ by definition. Substituting \refeqn{uiperp} into \refeqn{widotf0}, and rearranging by the facts that the matrix $-\hat q_i^2$ corresponds to the orthogonal projection to the plane normal to $q_i$ and $\hat q_i^3=-\hat q_i$, we obtain
\begin{align}
\dot\omega_i & = -k_q e_{q_i} -k_{\omega}e_{\omega_i} -(q_i\cdot\omega_{i_d})\dot q_i -\hat q_i^2\dot\omega_d\nonumber\\
&\quad - \frac{1}{m_il_i}\hat q_i \Phi_{x_i}^\perp \tilde\theta_{x_i},\label{eqn:dotwif}
\end{align}
where the estimation error is defined as $\tilde\theta_{x_i}^\perp=\theta_{x_i}-\bar\theta_{x_i}\in\Re^{n_\theta}$.

In short, the control force for the simplified dynamic model is given by
\begin{align}
u_i = u_i^\parallel + u_i^\perp.\label{eqn:ui}
\end{align}

\subsection{Design of Adaptive Law}\label{sec:AL}

Next, we design the adaptive laws to construct the estimates of unknown parameters. The following projection operator is introduced such that the estimated parameters stay in the bound of the true parameters given by \refeqn{Btheta}.
\begin{align}
\mathrm{Pr}(\bar\theta,y) = \begin{cases}
y & \text{if  $\|\bar\theta\|<B_\theta$}\\
  &\hspace*{-0.15\columnwidth} \text{or $\|\bar\theta\|=B_\theta$ and $\bar\theta^Ty \leq 0$},\\
(I_{n_\theta\times n_\theta}-\frac{1}{\|\bar\theta\|^2}\bar\theta\bar\theta^T)y & \text{otherwise}.\\
\end{cases}\label{eqn:Pr}
\end{align}
Using this, the adaptive laws are defined as
\begin{align}
\dot{\bar \theta}_{x_0} = \mathrm{Pr}(\bar \theta_{x_0}, y_{x_0}),\label{eqn:hatDx0_dot}\\
\dot{\bar \theta}_{R_0} = \mathrm{Pr}(\bar \theta_{R_0}, y_{R_0}),\label{eqn:hatDR0_dot}\\
\dot{\bar \theta}_{x_i} = \mathrm{Pr}(\bar \theta_{x_i}, y_{x_i}),\label{eqn:hatDxi_dot}
\end{align}
where $y_{x_0},y_{R_0},y_{x_i}\in\Re^{n_\theta}$ are defined as
\begin{align}
y_{x_0} & = \frac{h_{x_0}}{m_0}\Phi_{x_0}^T(\dot e_{x_0}+c_x e_{x_0}),\label{eqn:y_x0}\\
y_{R_0} & = h_{R_0}\Phi_{R_0}^T(e_{\Omega_0}+ c_R e_{R_0}),\label{eqn:y_R0}\\
y_{x_i} & = h_{x_i}\Phi_{x_0}^T\big[q_iq_i^T\{\frac{1}{m_0}(\dot e_{x_0}+c_x e_{x_0})\nonumber\\
&\quad
-R_0\hat\rho_i(e_{\Omega_0}+c_R e_{R_0})\} +  \frac{1}{m_il_i}\hat q_i (e_{\omega_i} + c_q e_{q_i})\big ],\label{eqn:y_xi}
\end{align}
for positive constants $c_x,c_R,c_q\in\Re$ and adaptive gains $h_{x_0},h_{R_0},h_{x_i}\in\Re$. 

The first case of the projection map is the identity map, and the second case corresponds to the case that the estimated parameters are at the boundary of the region defined by \refeqn{Btheta} and the unprojected direction $y$ for the change of the estimates points outward. For such cases, $y$ is projected onto the plane tangent to the boundary such that the estimated parameters remain on the region~\cite{IoaSun95}.

The resulting stability properties are summarized as follows.

\begin{prop}\label{prop:SDM}
Consider the simplified dynamic model defined by \refeqn{ddotx0}-\refeqn{dotwi}. For given tracking commands $x_{0_d},R_{0_d}$, a control input is designed as \refeqn{ui}-\refeqn{hatDxi_dot}. Then, there exist the values of controller gains and controller parameters such that the following properties are satisfied.

\begin{itemize}
\item[(i)] The zero equilibrium of tracking errors $(e_{x_0},\dot e_{x_0},e_{R_0},e_{\Omega_0},e_{q_i},e_{\omega_i})$ and the estimation errors $(\bar\theta{x_0},\bar\theta_{R_0},\bar\theta_{x_i})$ is stable in the sense of Lyapunov.
\item[(ii)] The tracking errors asymptotically coverage to zero.
\item[(iii)] The magnitude of the estimated parameters is less than $B_\theta$ always, provided that the magnitude of their initial estimates is less than $B_\theta$. 
\end{itemize}
\end{prop}
\begin{proof}
Due to the page limit, the proof is relegated to~\cite{LeeACC151ext}.
\end{proof}



\section{Control System Design for Full Dynamic Model}\label{sec:FDM}

The control system designed at the previous section is based on a simplifying assumption that each quadrotor can generate a thrust along any arbitrary direction instantaneously. However, the dynamics of quadrotor is underactuated since the direction of the total thrust is always parallel to its third body-fixed axis, while the magnitude of the total thrust can be arbitrarily changed. This can be directly observed from the expression of the total thrust, $u_i = -f_i R_i e_3$, where $f_i$ is the total thrust magnitude, and $R_ie_3$ corresponds to the direction of the third body-fixed axis. Whereas, the rotational attitude dynamics is fully actuated by the control moment $M_i$.

Based on these observations, the attitude of each quadrotor is controlled such that the third body-fixed axis becomes parallel to the direction of the ideal control force $u_i$ designed in the previous section within a finite time. More explicitly, the desired attitude of each quadrotor is constructed as follows. The desired direction of the third body-fixed axis of the $i$-th quadrotor, namely $b_{3_i}\in\Sph^2$ is given by
\begin{align}
b_{3_i} = -\frac{u_i}{\|u_i\|}.\label{eqn:b3i}
\end{align}
This provides two-dimensional constraint on the three-dimensional desired attitude of each quadrotor, and there remains one degree of freedom. To resolve it, the desired direction of the first body-fixed axis $b_{1_i}(t)\in\Sph^2$ is introduced as a smooth function of time~\cite{LeeLeoPICDC10}. This corresponds to controlling the additional one dimensional yawing angle of each quadrotor. From these, the desired attitude of the $i$-th quadrotor is given by
\begin{align*}
R_{i_c} = \begin{bmatrix}-\frac{(\hat b_{3_i})^2 b_{1_i}}{\|(\hat b_{3_i})^2 b_{1_i}\|}, &
\frac{\hat b_{3_i}b_{1_i}}{\|\hat b_{3_i}b_{1_i}\|},& b_{3_i}\end{bmatrix},
\end{align*}
which is guaranteed to be an element of $\SO$. The desired angular velocity is obtained from the attitude kinematics equation, $\Omega_{i_c} = (R_{i_c}^T\dot R_{i_c})^\vee\in\Re^3$.  

In the prior work described in~\cite{LeePICDC14}, the attitude of each quadrotor is controlled such that the equilibrium $R_i=R_{i_c}$ becomes exponentially stable, and the stability of the combined full dynamic model is achieved via singular perturbation theory~\cite{Kha96}. However, we can not follow such approach in this paper, as the presented adaptive control system guarantees only the asymptotical convergence of the tracking error variables due to the disturbances, thereby making is challenging to apply the singular perturbation theory. Here, we design the attitude controller of each quadrotor such that $R_i$ becomes equal to $R_{i_c}$ within a finite time via finite-time stability theory~\cite{BhaBerSJCO00,YuYuA05,WuRadAA11}. 

Define the tracking error vectors $e_{R_i},e_{\Omega_i}\in\Re^3$ for the attitude and the angular velocity of the $i$-th quadrotor as
\begin{align*}
e_{R_i} = \frac{1}{2}(R_{i_c}^T R_i -R_i^T R_{i_c})^\vee,\quad
e_{\Omega_i} = \Omega_i - R_i^T R_{i_c}\Omega_{i_c}.
\end{align*}
The time-derivative of $e_{R_i}$ can be written as~\cite{LeeLeoPICDC10}
\begin{align}
\dot e_{R_i} & = \frac{1}{2}(\trs{R_i^T R_{i_c}}I-R_i^T R_{i_c})e_{\Omega_i}\triangleq E(R_i,R_{i_c})e_{\Omega_i}.\label{eqn:eRi_dot}
\end{align}
For $0< r <1$, define $S:\Re\times\Re^3\rightarrow\Re^3$ as
\begin{align*}
S(r,y) = \begin{bmatrix}|y_1|^r\mathrm{sgn}(y_1),&
|y_2|^r\mathrm{sgn}(y_2),&
|y_3|^r\mathrm{sgn}(y_3)\end{bmatrix}^T,
\end{align*}
where $y=[y_1,y_2,y_3]^T\in\Re^3$, and $\mathrm{sgn}(\cdot)$ denotes the sign function. For positive constants $k_{R},l_{R}$, the terminal sliding surface $s_i\in\Re^3$ is designed as
\begin{align}
s_i = e_{\Omega_i} + k_{R} e_{R_i} + l_{R}S(r,e_{R_i}).\label{eqn:si}
\end{align}
We can show that when confined to the surface of $s_i\equiv 0$, the tracking errors become zero in a finite time. To reach the sliding surface, for positive constants $k_s,l_s$, the control moment is designed as
\begin{align}
M_i & = -k_s s_i - l_s S(r,s_i)+\Omega_i\times J_i\Omega_i\nonumber\\
&\quad - (k_R J_i+ l_s r J_i\mathrm{diag}_j[|e_{R_{i_j}}|^{r-1}]) E(R_i,R_{c_i}) e_{\Omega_i}\nonumber\\
&\quad  -J_i(\hat\Omega_i R_i^T R_{i_c} \Omega_{i_c} - R_i^T R_{i_c}\dot\Omega_{i_c}).\label{eqn:Mi}
\end{align}

The thrust magnitude is chosen as the length of $u_i$, projected on to $-R_ie_3$,
\begin{align}
f_i & = - u_i\cdot R_i e_3,\label{eqn:fi}
\end{align}
which yields that the thrust of each quadrotor becomes equal to its desired value $u_i$ when $R_i=R_{i_c}$.

Stability of the corresponding controlled systems for the full dynamic model can be shown by using the fact that the full dynamic model becomes exactly same as the simplified dynamic model within a finite time.

\begin{prop}\label{prop:FDM}
Consider the full dynamic model defined by \refeqn{ddotx0}-\refeqn{dotWi}. For given tracking commands $x_{0_d},R_{0_d}$ and the desired direction of the first body-fixed axis $b_{1_i}$, control inputs for quadrotors are designed as \refeqn{Mi} and \refeqn{fi}. Then, there exists controller parameters such that the tracking error variables 
$(e_{x_0},\dot e_{x_0},e_{R_0},e_{\Omega_0},e_{q_i},e_{\omega_i})$ asymptotically converge to zero, and the estimation errors are uniformly bounded.
\end{prop}
\begin{proof}
See Appendix \ref{sec:prfFDM}.
\end{proof}

This implies that the payload asymptotically follows any arbitrary desired trajectory both in translations and rotations in the presence of uncertainties. In contrast to the existing results in aerial transportation of a cable suspended load, it does not rely on any simplifying assumption that ignores the coupling between payload, cable, and quadrotors. Also, the presented global formulation on the nonlinear configuration manifold avoids singularities and complexities that are inherently associated with local coordinates. As such, the presented control system is particularly useful for agile load transportation involving combined translational and rotational maneuvers of the payload in the presence of uncertainties.

\section{Numerical Example}\label{sec:NE}
We consider a numerical example where three quadrotors ($n=3$) transport a rectangular box along a figure-eight curve. More explicitly, the mass of the payload is $m_0=1.5\,\mathrm{kg}$, and its length, width, and height are $1.0\,\mathrm{m}$, $0.8\,\mathrm{m}$, and $0.2\,\mathrm{m}$, respectively. Mass properties of three quadrotors are identical, and they are given by
\begin{gather*}
m_i=0.755\,\mathrm{kg},\quad J_i=\mathrm{diag}[0.0820,\, 0.0845,\, 0.1377]\,\mathrm{kgm^2}.
\end{gather*}
The length of cable is $l_i=1\,\mathrm{m}$, and they are attached to the following points of the payload.
\begin{gather*}
\rho_1 = [0.5,\,0,\,-0.1]^T,\\
\rho_1 = [-0.5,\,0.4,\,-0.1]^T,\quad
\rho_3 = [-0.5,\,-0.4,\,-0.1]^T\,\mathrm{m}.
\end{gather*}
In other words, the first link is attached to the center of the top, front edge, and the remaining two links are attached to the vertices of the top, rear edge (see Figure \ref{fig:DM}). 

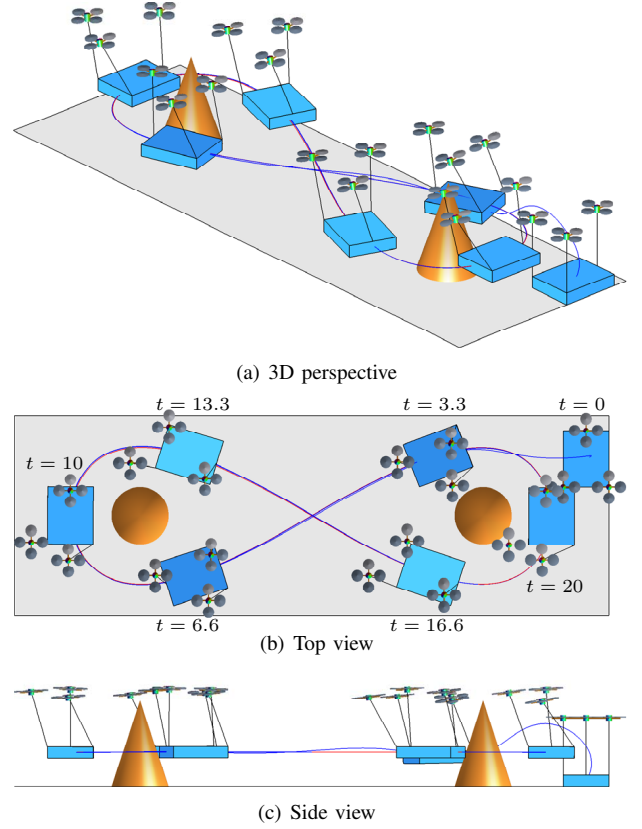
\begin{figure}
\centerline{
	\subfigure[3D perspective]{
		\includegraphics[width=0.95\columnwidth]{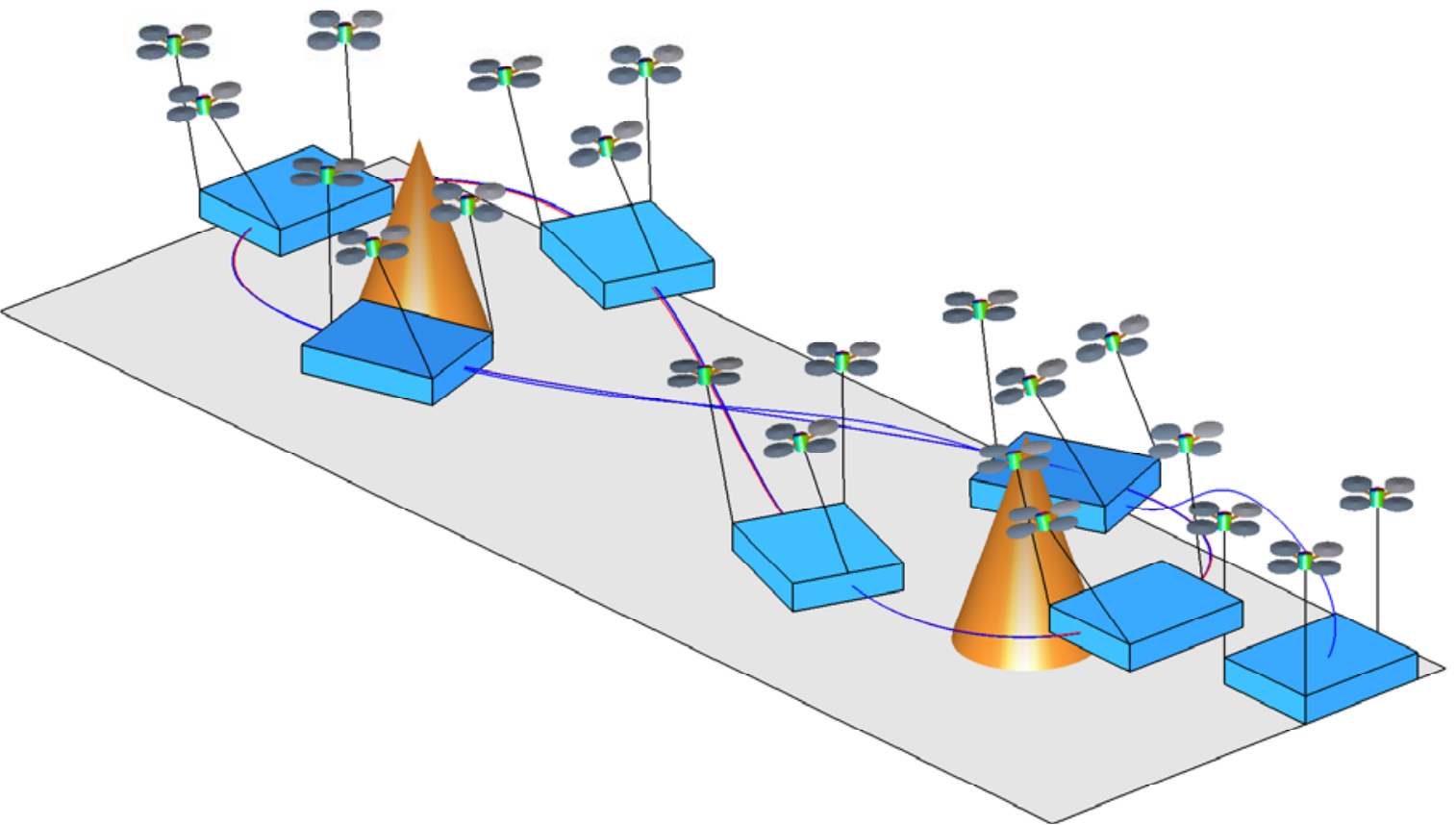}}
}
\centerline{
	\subfigure[Top view]{
	\setlength{\unitlength}{0.1\columnwidth}\scriptsize
	\begin{picture}(9.5,3.2)(0,0)
	\put(0,0){\includegraphics[width=0.95\columnwidth]{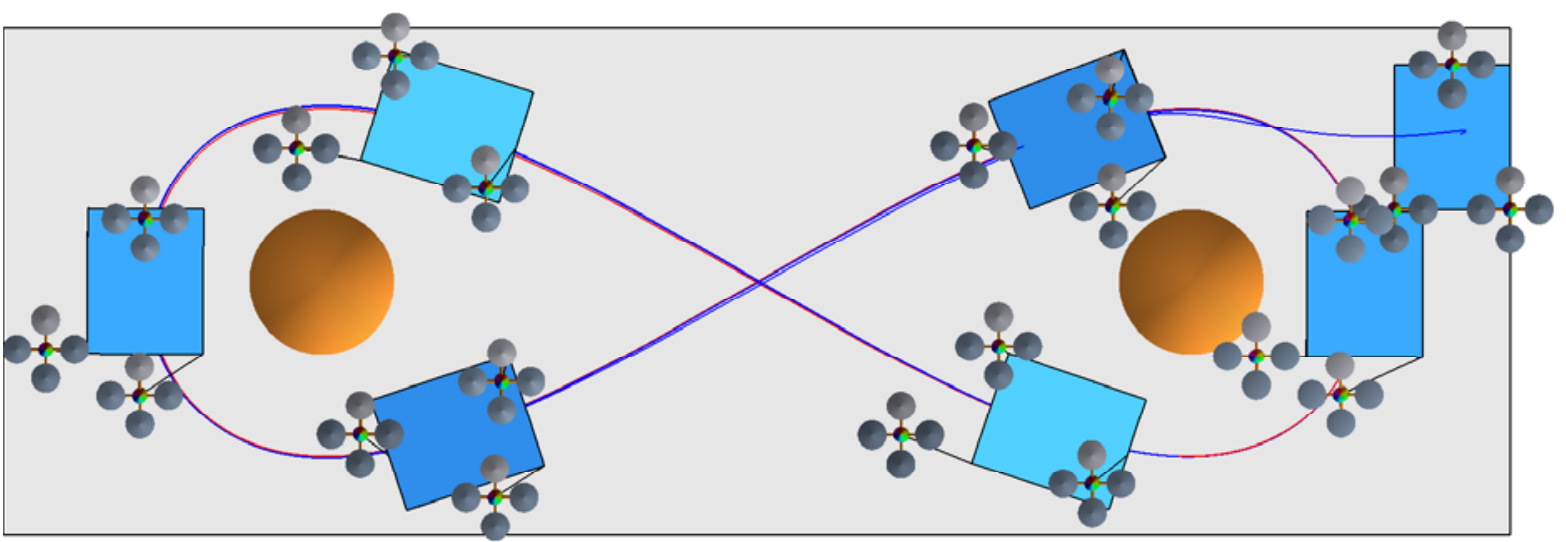}}
	\put(8.4,3.2){$t=0$}
	\put(6.0,3.2){$t=3.3$}
	\put(2.2,-0.2){$t=6.6$}
	\put(0.2,2.3){$t=10$}
	\put(2.2,3.2){$t=13.3$}
	\put(5.8,-0.2){$t=16.6$}
	\put(7.9,0.4){$t=20$}
	\end{picture}}	
}
\centerline{
	\subfigure[Side view]{
		\includegraphics[width=0.95\columnwidth]{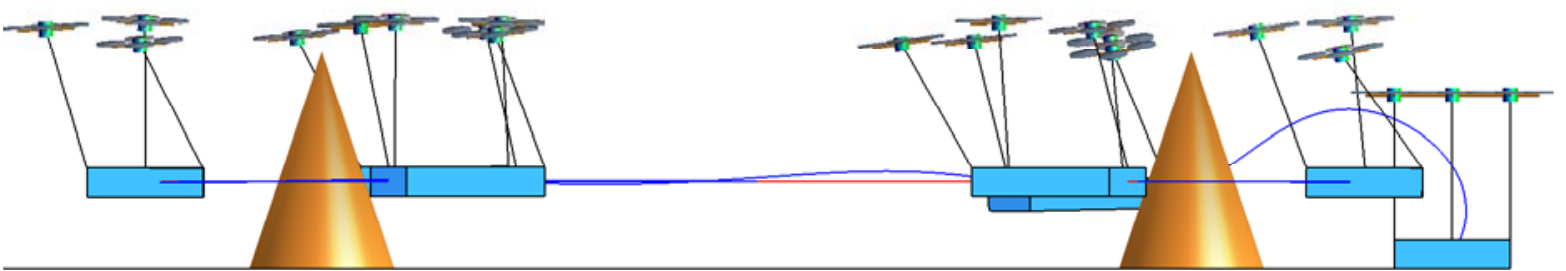}}
}
\caption{Snapshots of controlled maneuver (red:desired trajectory, blue:actual trajectory). A short animation illustrating this maneuver is available at {\href{http://youtu.be/nOWErfdzZLU}{http://youtu.be/nOWErfdzZLU}}.}\label{fig:SS}
\end{figure}

The desired trajectory of the payload is chosen as
\begin{align*}
x_{0_d}(t) = [1.2\sin (0.2\pi t),\, 4.2\cos (0.1\pi t),\, -0.5]^T\,\mathrm{m}.
\end{align*}
The desired attitude of the payload is chosen such that its first axis is tangent to the desired path, and the third axis is parallel to the direction of gravity, it is given by
\begin{align*}
R_{0_d}(t) = \begin{bmatrix}\frac{\dot x_{0_d}}{\|\dot x_{0_d}\|}&\frac{\hat e_3\dot x_{0_d}}{\|\hat e_3\dot x_{0_d}\|}& e_3\end{bmatrix}.
\end{align*}
Initial conditions are chosen as
\begin{gather*}
x_0(0)=[1,\,4.8,\,0]^T\,\mathrm{m},\quad v_0(0)=0_{3\times 1}\,\mathrm{m/s},\\
q_i(0)= e_3,\; \omega_i(0)=0_{3\times 1},\;
R_i(0)=I_{3\times 3},\; \Omega_i(0)=0_{3\times 1}.
\end{gather*}
The uncertainties are specified as
\begin{align*}
\Delta_{x_0}=[1,\, 3,\, -2.5]^T,\quad \Delta_{R_0}=[-0.5,\, 0.1,\, -1.5]^T,\\
\Delta_{x_i}=[0.5,\, -0.2,\, 0.3]^T,\quad \Delta_{R_i}=[0.2,\, 0.3,\, -0.7]^T.
\end{align*}

\begin{figure}
\centerline{
	\subfigure[Position of payload ($x_0$:blue, $x_{0_d}$:red)]{
		\includegraphics[width=0.48\columnwidth]{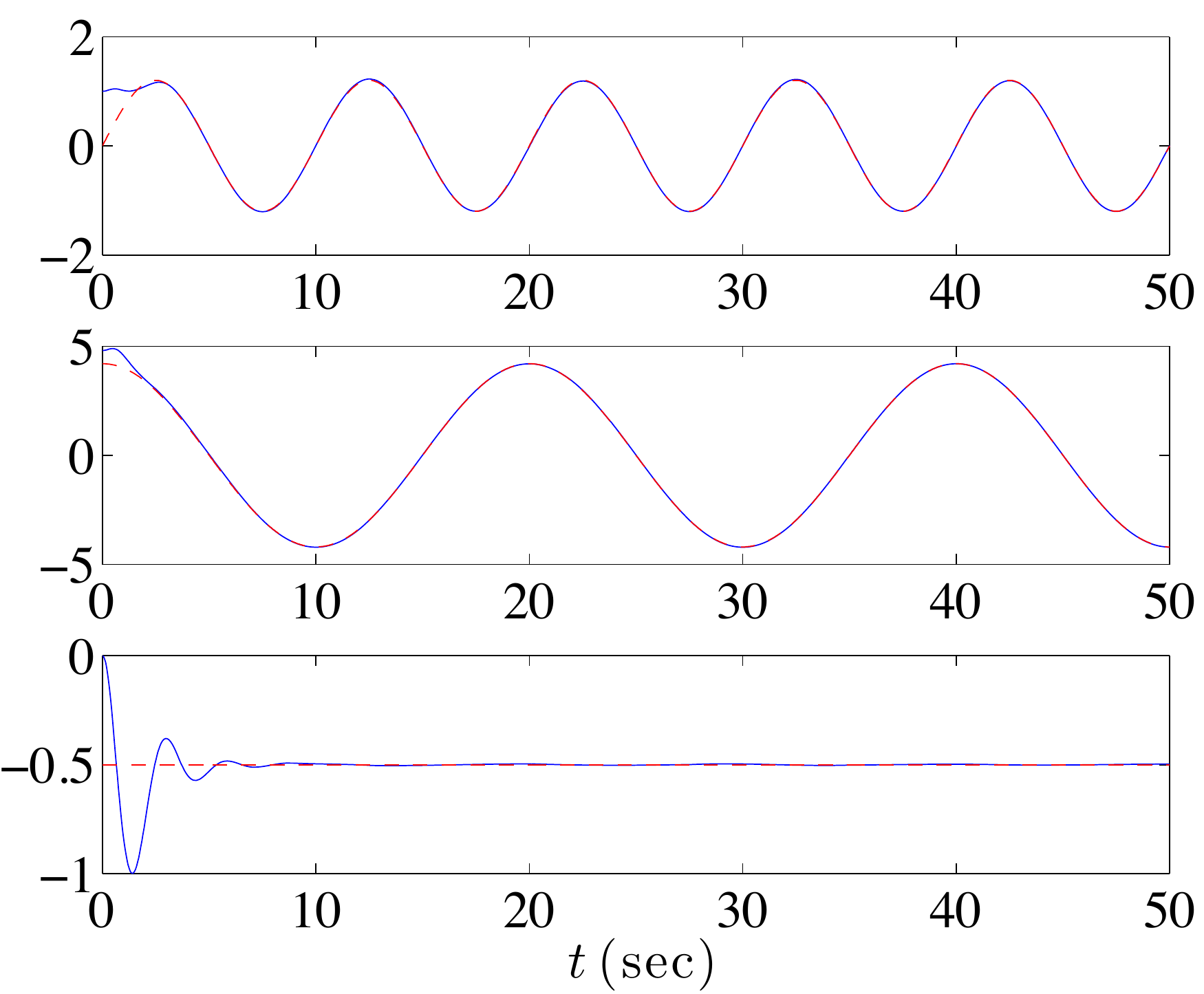}}
	\hfill
	\subfigure[Attitude tracking error of payload $\Psi_0=\frac{1}{2}\|R_0-R_{0_d}\|^2$]{
		\includegraphics[width=0.48\columnwidth]{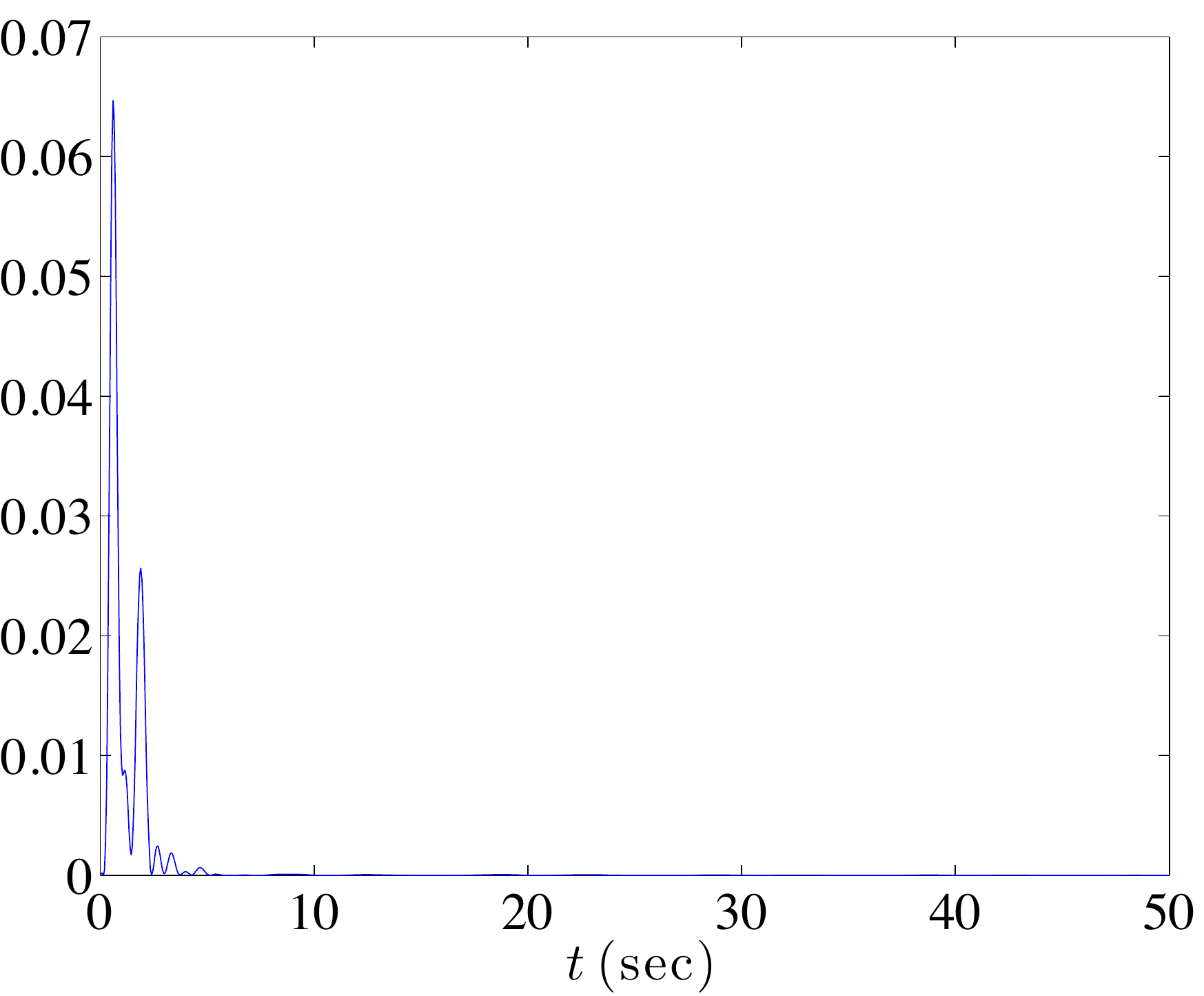}}
}
\centerline{
	\subfigure[Link direction error $\Psi_{q_i}=\frac{1}{2}\|q_i-q_{i_d}\|^2$]{
		\includegraphics[width=0.48\columnwidth]{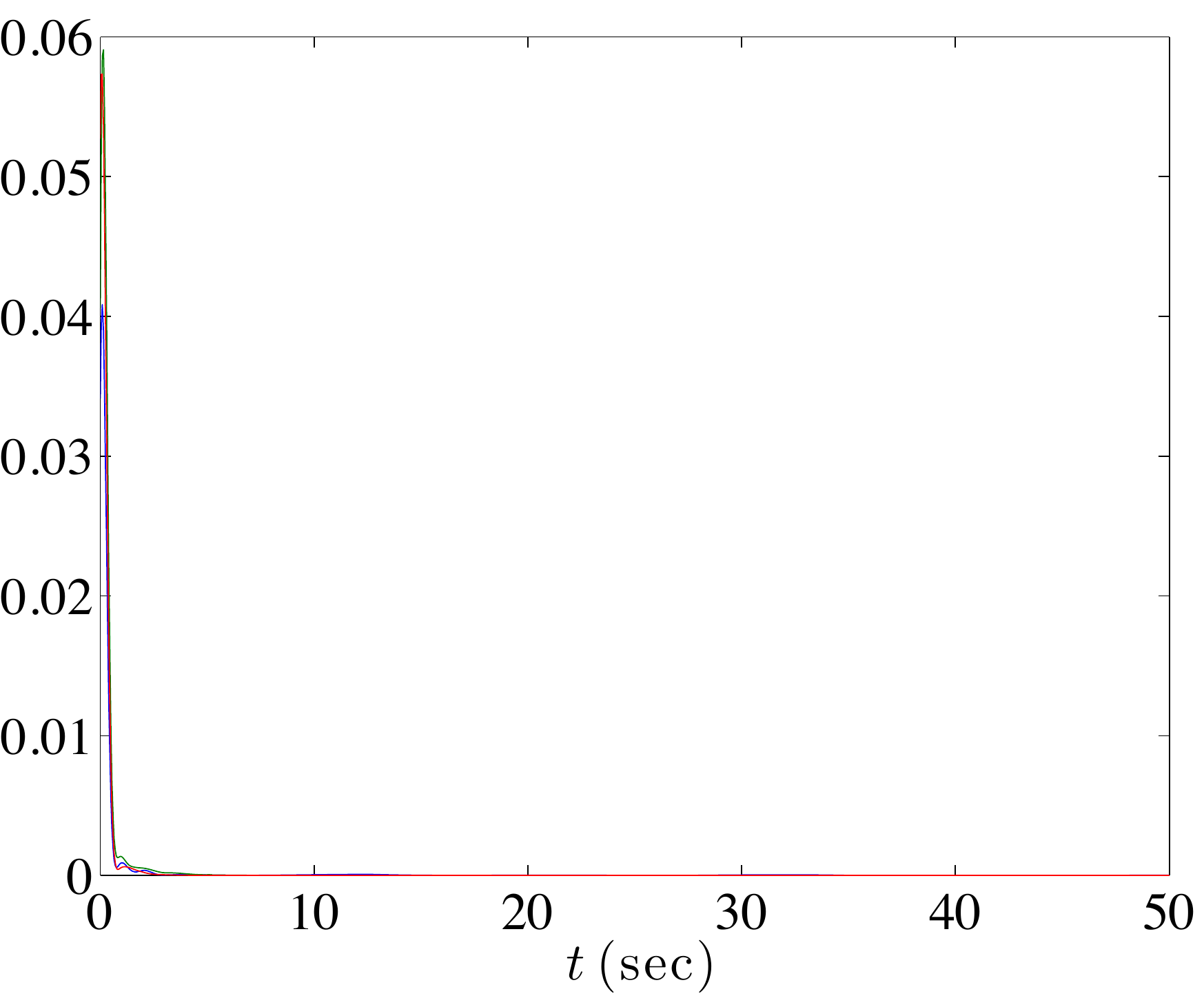}}
	\hfill
	\subfigure[Attitude tracking error of quadrotors $\Psi_i=\frac{1}{2}\|R_i-R_{i_d}\|^2$]{
		\includegraphics[width=0.48\columnwidth]{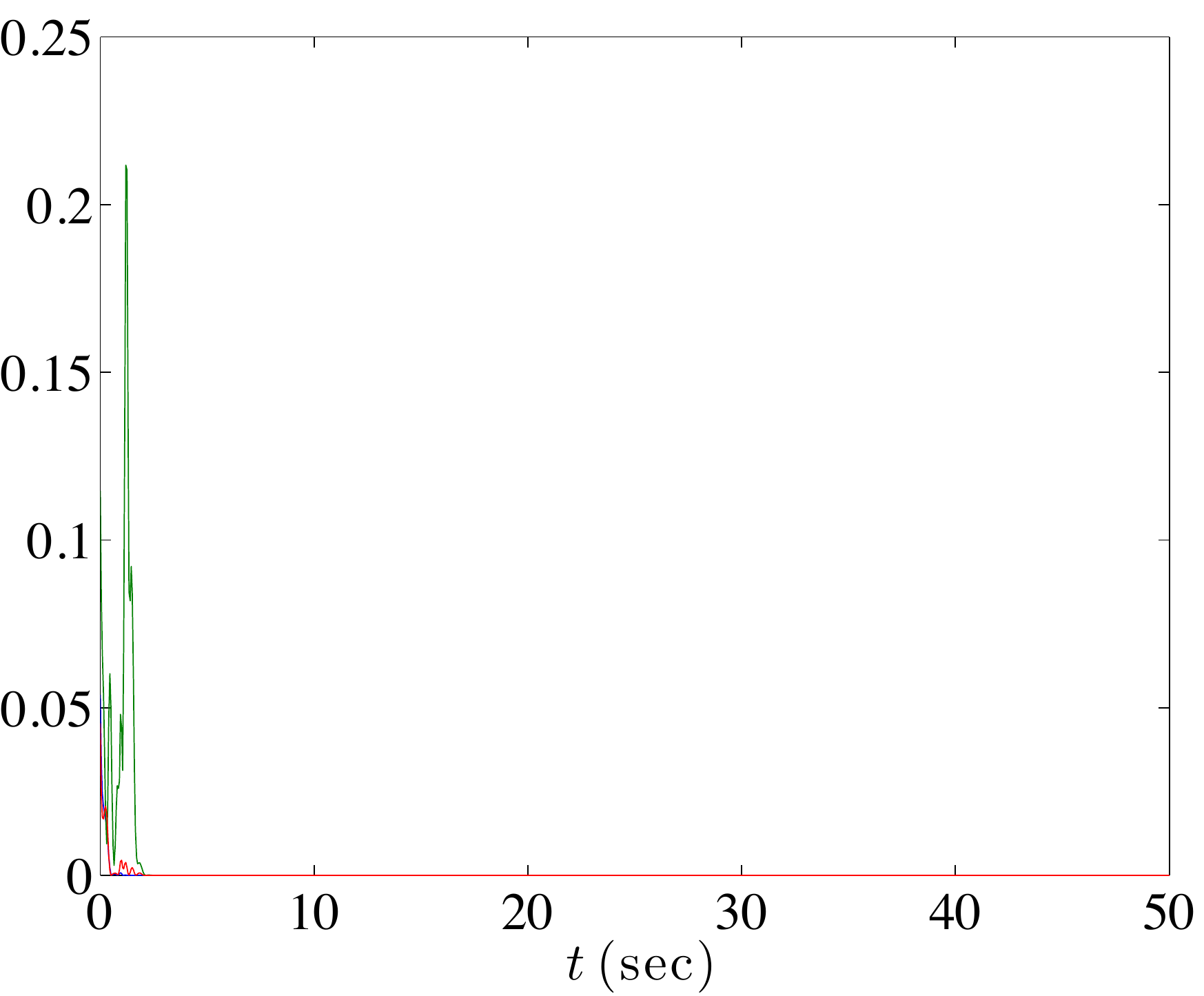}}
}
\centerline{
	\subfigure[Tension at links]{
		\includegraphics[width=0.45\columnwidth]{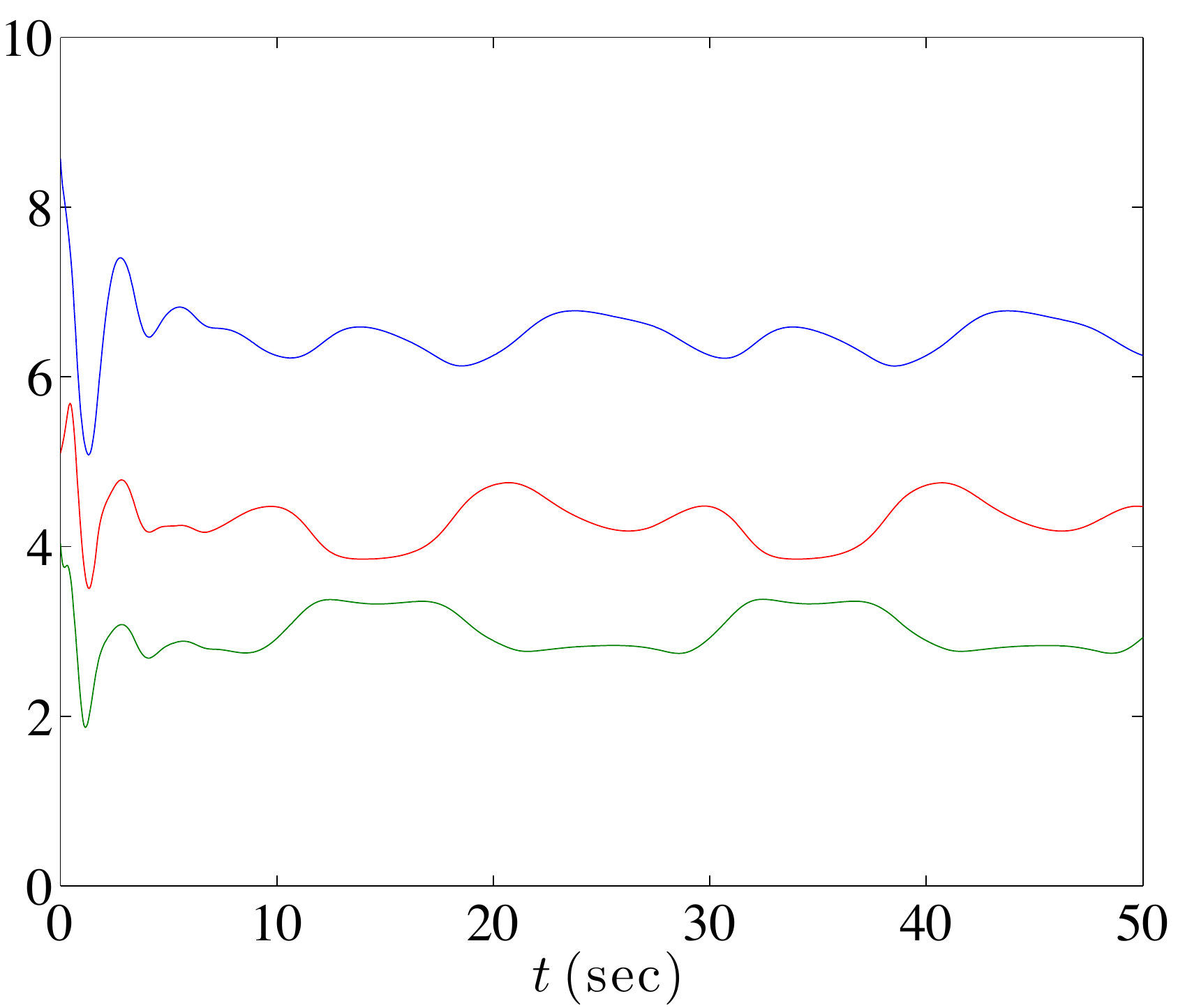}}
	\hfill
	\subfigure[Control input for quadrotors $f_i,M_i$]{
		\includegraphics[width=0.48\columnwidth]{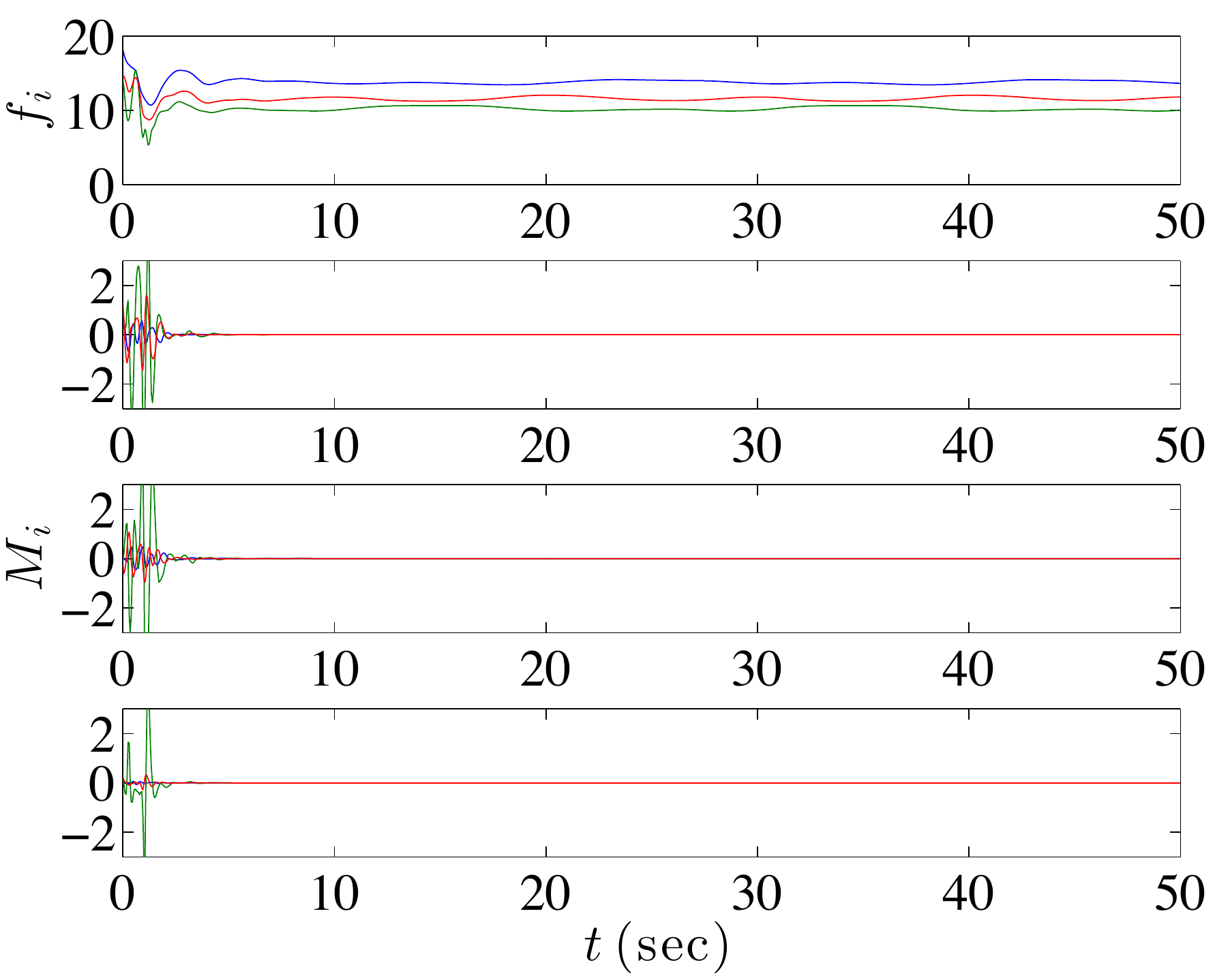}}
}
\caption{Simulation results for tracking errors and control inputs. (for figures (c)-(f): $i=1$:blue, $i=2$:green, $i=3$:red)}\label{fig:SR}
\end{figure}

The corresponding simulation results are presented at Figures \ref{fig:SS} and \ref{fig:SR}. Figure \ref{fig:SS} illustrates the desired trajectory that is shaped like a figure-eight curve around two obstacles represented by cones, and the actual maneuver of the payload and quadrotors. Figure \ref{fig:SR} shows tracking errors for the position and the attitude of the payload, tracking errors for the link directions and the attitude of quadrotors, as well as tension and control inputs. These illustrate excellent tracking performances of the proposed control system.

\appendix
\paragraph{Error Dynamics}

From \refeqn{ddotx0s} and \refeqn{mui}, the dynamics of the position tracking error is given by
\begin{align*}
m_0\ddot e_{x_0} & = m_0(ge_3-\ddot x_{0_d}) + \Delta_{x_0} + \sum_{i=1}^n (q_iq_i^T \mu_{i_d}+ \Delta_{x_i}^\parallel).
\end{align*}
From \refeqn{Fd} and \refeqn{muid0}, this can be rearranged as
\begin{align}
\ddot e_{x_0} & = ge_3-\ddot x_{0_d} + \frac{1}{m_0}F_d  +Y_x+\frac{1}{m_0}(\Delta_{x_0}+\sum_{i=1}^n \Delta_{x_i}^\parallel),\nonumber\\
& = -k_{x_0}e_{x_0} - k_{\dot x_0} \dot e_{x_0} +\frac{1}{m_0}(\Phi_{x_0}\tilde\theta_{x_0}+\sum_{i=1}^n \Phi_{x_i}^\parallel \tilde\theta_{x_i}) +Y_x,
\label{eqn:ddotex0}
\end{align}
where \refeqn{delxiparallel} is used and the estimation error is denoted by $\tilde\theta_{x_i}=\theta_{x_i}-\bar{\theta}_{x_i}$. At the above equation, the last term $Y_x\in\Re^3$ represents the error caused by the difference between $q_i$ and $q_{i_d}$, and it is given by
\begin{align*}
Y_x=\frac{1}{m_0}\sum_{i=1}^n (q_iq_i^T -I) \mu_{i_d}.
\end{align*}
We have $\mu_{i_d}=q_{i_d}q_{i_d}^T\mu_{i_d}$ from \refeqn{qid}. Using this, the error term can be written in terms of $e_{q_i}$ as
\begin{align*}
Y_x  & = \frac{1}{m_0}\sum_{i=1}^n 
 (q_{i_d}^T\mu_{i_d})\{(q_i^Tq_{i_d})q_i-q_{i_d}\}\nonumber\\
& = -\frac{1}{m_0}\sum_{i=1}^n (q_{i_d}^T\mu_{i_d}) \hat q_i e_{q_i}.
\end{align*}
Using \refeqn{muid}, an upper bound of $Y_x$ can be obtained as
\begin{align*}
\|Y_x\| & \leq \frac{1}{m_0} \sum_{i=1}^n \|\mu_{i_d}\|\|e_{q_i}\|\leq  \sum_{i=1}^n\gamma (\|F_d\|+\|M_d\|)\|e_{q_i}\|,
\end{align*}
where $\gamma= \frac{1}{m_0\sqrt{\lambda_{m}[\mathcal{P}\mathcal{P}^T]}}$. From \refeqn{Fd} and \refeqn{Md}, this can be further bounded by
\begin{align}
\|Y_x\|  \leq & \sum_{i=1}^n \{\beta(k_{x_0}\|e_{x_0}\|+k_{\dot x_0}\|\dot e_{x_0}\|)\nonumber\\
&\quad  + \gamma(k_{R_0} \|e_{R_0}\| + k_{\Omega_0}\|e_{\Omega_0}\|) + B\}\|e_{q_i}\|,\label{eqn:YxB}
\end{align}
where $\beta=m_0\gamma$, and the constant $B$ is determined by the given desired trajectories of the payload, the assumption \refeqn{BPhi} on the bounds of $\Phi$ terms, and the adaptive law defined later that guarantee the boundedness of the estimated parameters $\bar\theta$. Throughout the remaining parts of the proof, any bound that can be obtained from $x_{0_d},R_{0_d}$, \refeqn{BPhi}, or the adaptive law is denoted by $B$ for simplicity. In short, the position tracking error dynamics of the payload can be written as \refeqn{ddotex0}, where the error term is bounded by \refeqn{YxB}.

Similarly, we find the attitude tracking error dynamics for the payload as follows. Using \refeqn{dotW0s}, \refeqn{Md}, and \refeqn{mui}, the time-derivative of $J_0 e_{\Omega_0}$ can be written as
\begin{align}
J_0\dot e_{\Omega_0} & = (J_0e_{\Omega_0}+d)^\wedge e_{\Omega_0} - k_{R_0}e_{R_0}-k_{\Omega_0}e_{\Omega_0}\nonumber\\
&\quad + \Phi_{R_0}\tilde\theta_{R_0} + \sum_{i=1}^n \hat\rho_i R_0^T\Phi_{R_0}^\parallel\tilde\theta_{x_i}+Y_R,\label{eqn:doteW0}
\end{align}
where $d=(2J_0-\trs{J_0}I)R_0^T R_{0_d}\Omega_{0_d}\in\Re^3$~\cite{GooLeePECC13} that is bounded, and $\tilde\Delta_{R_0}\in\Re^{3}$ denotes the estimation error given by $\tilde\Delta_{R_0} = \Delta_{R_0}-\hat\Delta_{R_0}$. The error term in the attitude dynamics of the payload, namely $Y_R\in\Re^3$ is given by
\begin{align*}
Y_R & = \sum_{i=1}^n\hat\rho_i R_0^T (q_iq_i^T - I)\mu_{i_d}
= -\sum_{i=1}^n\hat\rho_i R_0^T (q_{i_d}^T \mu_{i_d}) \hat q_i e_{q_i}.
\end{align*}
Similar with \refeqn{YxB}, an upper bound of $Y_R$ can be obtained as
\begin{align}
\|Y_R\|  \leq & \sum_{i=1}^n\{\delta_i (k_{x_0}\|e_{x_0}\|+k_{\dot x_0}\|\dot e_{x_0}\|)\nonumber\\
&\quad  + \sigma_i(k_{R_0} \|e_{R_0}\| + k_{\Omega_0}\|e_{\Omega_0}\|) + B\}\|e_{q_i}\|,\label{eqn:YRB}
\end{align}
where $\delta_i = m_0\frac{\|\hat\rho_i\|}{\sqrt{\lambda_{m}[\mathcal{P}\mathcal{P}^T]}},\sigma_i=\frac{\delta_i}{m_0}\in\Re$.

Next, from \refeqn{dotwif}, the time-derivative of the angular velocity error, projected on to the plane normal to $q_i$ is given as
\begin{align}
-\hat q_i^2 \dot e_{\omega_i} 
& = \dot\omega +( q\cdot\omega_d) \dot q +\hat q^2 \dot \omega_d\nonumber\\
& = -k_q e_{q_i} - k_\omega e_{\omega_i}-\frac{1}{m_il_i}\hat q_i \Phi_{x_i}^\perp\tilde\theta_{x_i}.\label{eqn:dotewi}
\end{align}
In summary, the error dynamics of the simplified dynamic model are given by \refeqn{ddotex0}, \refeqn{doteW0} and \refeqn{dotewi}.

\paragraph{Stability Proof}

Define an attitude configuration error function $\Psi_{R_0}$ for the payload as
\begin{align*}
\Psi_{R_0} = \frac{1}{2}\trs{I-R_{0_d}^T R_0},
\end{align*}
which is positive-definite about $R_0=R_{0_d}$, and $\dot\Psi_{R_0} = e_{R_0}\cdot e_{\Omega_0}$~\cite{LeeLeoPICDC10,GooLeePECC13}. We also introduce a configuration error function $\Psi_{q_i}$ for each link that is positive-definite about $q_i=q_{i_d}$ as
\begin{align*}
\Psi_{q_i} = 1-q_{i}\cdot q_{i_d}.
\end{align*}
For positive constants $e_{x_{\max}}, \psi_{R_0}, \psi_{q_i}$, consider the following open domain containing the zero equilibrium of tracking error variables:
\begin{align}
D=\{&(e_{x_0},\dot e_{x_0}, e_{R_0}, e_{\Omega_0}, e_{q_i}, e_{\omega_i},\tilde\Delta_{x_0},\tilde\Delta_{R_0},\tilde\Delta_{x_i})\nonumber\\
& \in(\Re^3)^4\times (\Re^3\times\Re^3)^n\times (\Re^3)^2\times \Re^{3n}\,|\,\nonumber\\
& \|e_{x_0}\|< e_{x_{\max}},\,\Psi_{R_0}< \psi_{R_0}<1,\,\Psi_{q_i}< \psi_{q_i}<1\}.\label{eqn:D}
\end{align}
In this domain, we have $\|e_{R_0}\|=\sqrt{\Psi_{R_0}(2-\Psi_{R_0})} \leq \sqrt{\psi_{R_0}(2-\psi_{R_0})} \triangleq \alpha_0 < 1$, and $\|e_{q_i}\|=\sqrt{\Psi_{q_i}(2-\Psi_{q_i})} \leq \sqrt{\psi_{q_i}(2-\psi_{q_i})} \triangleq \alpha_i < 1$. It is assumed that $\psi_{q_i}$ is sufficiently small such that $n\alpha_i\beta < 1$. 

We can show that the configuration error functions are quadratic with respect to the error vectors in the sense that
\begin{gather*}
\frac{1}{2}\|e_{R_0}\|^2 \leq \Psi_{R_0} \leq \frac{1}{2-\psi_{R_0}} \|e_{R_0}\|^2,\\
\frac{1}{2}\|e_{q_i}\|^2 \leq \Psi_{q_i} \leq \frac{1}{2-\psi_{q_i}} \|e_{q_i}\|^2,
\end{gather*}
where the upper bounds are satisfied only in the domain $D$. 

Define
\begin{align*}
\mathcal{V}_0 & = \frac{1}{2}\|\dot e_{x_0}\|^2 + \frac{1}{2}k_{x_0}\|e_{x_0}\|^2   +c_x e_{x_0}\cdot \dot e_{x_0}\\
&\quad + \frac{1}{2}e_{\Omega_0}\cdot J_0\Omega_0 + k_{R_0}\Psi_{R_0} + c_R e_{R_0}\cdot J_0e_{\Omega_0}\\
&\quad+ \sum_{i=1}^n \frac{1}{2}\|e_{\omega_i}\|^2 + k_q\Psi_{q_i} + c_q e_{q_i}\cdot e_{\omega_i},
\end{align*}
where $c_x,c_R,c_q$ are positive constants. This is composed of tracking error variables only, and we define another function for the estimation errors of the adaptive laws as
\begin{align*}
\mathcal{V}_a & = \frac{1}{2h_{x_0}}\|\tilde \theta_{x_0}\|^2
+\frac{1}{2h_{R_0}}\|\tilde \theta_{R_0}\|^2
+ \sum_{i=1}^n \frac{1}{2h_{x_i}} \|\tilde\theta_{x_i}\|^2.
\end{align*}
The Lyapunov function for the complete simplified dynamic model is chosen as $\mathcal{V}=\mathcal{V}_0+\mathcal{V}_a$.

Let $z_{x_0} = [\|e_{x_0}\|,\|\dot e_{x_0}\|]^T$, $z_{R_0}=[\|e_{R_0}\|,\|e_{\Omega_0}\|]^T$, $z_{q_i}= [\|e_{q_i}\|,\|e_{\omega_i}\|]^T\in\Re^2$. The first part of the Lyapunov function $\mathcal{V}_0$ satisfies
\begin{align*}
z_{x_0}^T &\underline{P}_{x_0}z_{x_0} + 
z_{R_0}^T \underline{P}_{R_0}z_{R_0} + 
\sum_{i=1}^n z_{q_i}^T \underline{P}_{q_i}z_{q_i}\leq\mathcal{V}_0\\
&\leq z_{x_0}^T \overline{P}_{x_0}z_{x_0} + 
z_{R_0}^T \overline{P}_{R_0}z_{R_0} + 
\sum_{i=1}^n z_{q_i}^T \overline{P}_{q_i}z_{q_i},
\end{align*}
where the matrices $\underline{P}_{x_0},\underline{P}_{R_0},\underline{P}_{q_i},\underline{P}_{x_0},\underline{P}_{R_0},\underline{P}_{q_i}\in\Re^{2\times 2}$ are given by
\begin{alignat*}{2}
\underline{P}_{x_0}&=\frac{1}{2}\begin{bmatrix} k_{x_0} & -c_x\\-c_x & 1\end{bmatrix},&\;
\overline{P}_{x_0}&=\frac{1}{2}\begin{bmatrix} k_{x_0} & c_x\\c_x & 1\end{bmatrix},\\
\underline{P}_{R_0}&=\frac{1}{2}\begin{bmatrix} 2k_{R_0} & -c_R\overline\lambda\\-c_R\overline\lambda & \underline\lambda\end{bmatrix},&
\overline{P}_{R_0}&=\frac{1}{2}\begin{bmatrix} \frac{2k_{R_0}}{2-\psi_{R_0}} & c_R\overline\lambda\\c_R\overline\lambda & \overline\lambda\end{bmatrix},\\
\underline{P}_{q_i}&=\frac{1}{2}\begin{bmatrix} 2k_{q} & -c_q\\-c_q & 1\end{bmatrix},&
\overline{P}_{q_i}&=\frac{1}{2}\begin{bmatrix} \frac{2k_{q}}{2-\psi_{q_i}} & c_q\\c_q & 1\end{bmatrix},
\end{alignat*}
where $\underline\lambda=\lambda_{m}[J_0]$ and $\overline\lambda=\lambda_{M}[J_0]$. If the constants $c_x,c_{R_0},c_q$ are sufficiently small, all of the above matrices are positive-definite. As the second part of the Lyapunov function $\mathcal{V}_a$ is already given as a quadratic form, it is straightforward to see that the complete Lyapunov function $\mathcal{V}$ is positive-definite and decrescent. 

The time-derivative of the Lyapunov function along the error dynamics \refeqn{ddotex0}, \refeqn{doteW0}, and \refeqn{dotewi} is given by
\begin{align}
\dot{\mathcal{V}} & = -(k_{\dot x_0}-c_x) \|\dot e_{x_0}\|^2 
-c_xk_{x_0} \|e_{x_0}\|^2 -c_xk_{\dot x_0} e_{x_0}\cdot \dot e_{x_0}\nonumber\\
&\quad+ (c_x e_{x_0}+\dot e_{x_0})\cdot Y_x -k_{\Omega_0}\|e_{\Omega_0}\|^2 + c_R \dot e_R\cdot J_0e_{\Omega_0}\nonumber\\
&\quad -c_R k_{R_0}\|e_{R_0}\|^2 + c_R e_{R_0}\cdot ((J_0e_{\Omega_0}+d)^\wedge e_{\Omega_0}-k_{\Omega_0} e_{\Omega_0})\nonumber\\
&\quad + (e_{\Omega_0}+c_R e_{R_0}) \cdot Y_R\nonumber\\
&\quad +\sum_{i=1}^n -(k_\omega-c_q) \|e_{\omega_i}\|^2 -c_qk_q \|e_{q_i}\|^2 -c_q k_\omega e_{q_i}\cdot e_{\omega_i}\nonumber\\
&\quad + \frac{1}{m_0}(\dot e_{x_0}+c_xe_{x_0})\cdot (\Phi_{x_0}\tilde\theta_{x_0} +  \sum_{i=1}^n \Phi_{x_i}^\parallel\tilde\theta_{x_i}) \nonumber\\
&\quad +(e_{\Omega_0}+ c_Re_{R_0})\cdot (\Phi_{R_0}\tilde\theta_{R_0}+\sum_{i=1}^n \hat\rho_i R_0^T \Phi_{x_i}^\parallel\tilde\theta_{x_i})\nonumber\\
&\quad -\parenth{\sum_{i=1}^n ( e_{\omega_i} + c_qe_{q_i})\cdot \frac{\hat q_i}{m_il_i}\Phi_{x_i}^\perp \tilde\theta_{x_i}}-\frac{1}{h_{x_0}}\tilde\theta_{x_0}\cdot \dot{\bar\theta}_{x_0}\nonumber\\
&\quad 
-\frac{1}{h_{R_0}}\tilde\theta_{R_0}\cdot\dot{\bar\theta}_{R_0}
-\sum_{i=1}^n\frac{1}{h_{x_i}}\tilde\theta_{x_i}\cdot \dot{\bar\theta}_{x_i}.
\label{eqn:dotV0}
\end{align}

In the above equation, the expressions at the last four lines depending on the estimate error can be rearranged by using the adaptive laws, \refeqn{hatDx0_dot}-\refeqn{y_xi} as
\begin{align*}
&\tilde\theta_{x_0}\cdot (y_{x_0}-\mathrm{Pr}(\bar\theta_{x_0},y_{x_0}))
+\tilde\theta_{R_0}\cdot (y_{R_0}-\mathrm{Pr}(\bar\theta_{R_0},y_{R_0}))\\
&\quad+\sum_{i=1}^n\tilde\theta_{x_i}\cdot (y_{x_i}-\mathrm{Pr}(\bar\theta_{x_i},y_{x_i})).
\end{align*}
From the definition of the projection map, the above expressions vanish for the first case of \refeqn{Pr}. For the second case, 
\begin{align*}
(\theta-\bar\theta)\cdot (y-\mathrm{Pr}(\bar\theta,y)) & = 
\frac{1}{\norm{\bar\theta}^2}(\theta-\bar\theta)\cdot \bar\theta\bar\theta^T y \\
& = \frac{1}{\norm{\bar\theta}^2}(\bar\theta^T\theta-\bar\theta^T\bar\theta) (\bar\theta^T y)\leq 0.
\end{align*}
for each estimated parameter. The last inequality is due to $(\bar\theta^T\theta-\bar\theta^T\bar\theta)\leq 0$ and $(\bar\theta^T y)>0$ obtained by \refeqn{Btheta} and \refeqn{Pr}.

An upper bound of the remaining expressions of $\dot{\mathcal{V}}$ at \refeqn{dotV0} can be obtained as follows.
Since $\|e_{R_0}\|\leq 1$, $\| \dot e_{R_0}\|\leq \|e_{\Omega_0}\|$ and $\|d\|\leq B$, 
\begin{align}
\dot{\mathcal{V}} & \leq -(k_{\dot x_0}-c_x) \|\dot e_{x_0}\|^2 
-c_xk_{x_0} \|e_{x_0}\|^2 -c_xk_{\dot x_0} e_{x_0}\cdot \dot e_{x_0}\nonumber\\
&\quad+ (c_x e_{x_0}+\dot e_{x_0})\cdot Y_x -(k_{\Omega_0}-2c_R\overline\lambda)\|e_{\Omega_0}\|^2 \nonumber\\
&\quad -c_R k_{R_0}\|e_{R_0}\|^2 + c_R(k_{\Omega_0}+B) \|e_{R_0}\|\|e_{\Omega_0}\|\nonumber\\
&\quad + (e_{\Omega_0}+c_R e_{R_0}) \cdot Y_R\nonumber\\
&\quad +\sum_{i=1}^n -(k_\omega-c_q) \|e_{\omega_i}\|^2 -c_qk_q \|e_{q_i}\|^2 -c_q k_\omega e_{q_i}\cdot e_{\omega_i}.
\label{eqn:dotV2}
\end{align}
From \refeqn{YxB}, an upper bound of the fourth term of the right-hand side is given by
\begin{align}
\|&(c_x e_{x_0}+\dot e_{x_0})\cdot Y_x\|  \leq  \nonumber\\
&\sum_{i=1}^n\alpha_i\beta( c_xk_{x_0}\|e_{x_0}\|^2 + c_x k_{\dot x_0}\|e_{x_0}\|\|\dot e_{x_0}\| +k_{\dot x_0} \|\dot e_{x_0}\|^2)
\nonumber\\
&+\{c_xB \|e_x\|+(\beta k_{x_0}e_{x_{\max}}+B) \|\dot e_{x_0}\| \} \|e_{q_i}\|\nonumber\\
&+\alpha_i\gamma(c_x\|e_{x_0}\|+\|\dot e_{x_0}\|)(k_{R_0}\|e_{R_0}\|+k_{\Omega_0}\|e_{\Omega_0}\|).
\label{eqn:YxB2}
\end{align}
Similarly, using \refeqn{YRB},
\begin{align}
\|&(c_R e_{R_0}+e_{\Omega_0})\cdot Y_R\|\leq \nonumber\\
&\sum_{i=1}^n \alpha_i\sigma_i( c_R k_{R_0}\|e_{R_0}\|^2 + c_R k_{\Omega_0}\|e_{R_0}\|\|e_{\Omega_0}\| + k_{\Omega_0}\|e_{\Omega_0}\|^2)\nonumber\\
&+\{c_R B\|e_{R_0}\| + (\alpha_0\sigma_ik_{R_0}+B)\|e_{\Omega_0}\|\}\|e_{q_i}\|\nonumber\\
&+\alpha_i\delta_i(c_R \|e_{R_0}\|+\|e_{\Omega_0}\|)(k_{x_0}\|e_{x_0}\|+k_{\dot x_0}\|\dot e_{x_0}\|).
\end{align}
Substituting these into \refeqn{dotV2} and rearranging, 
\begin{align}
\dot{\mathcal{V}} \leq \sum_{i=1}^n - z_i^T W_i z_i,\label{eqn:U0}
\end{align}
where $z_i=[\|z_{x_0}\|,\, \|z_{R_0}\|,\,\ \|z_{q_i}\|]^T\in\Re^3$, and the matrix $W_i\in\Re^{3\times 3}$ is defined as
\begin{align}
W_i = 
\begin{bmatrix} \lambda_m[W_{x_i}] & -\frac{1}{2}\|W_{xR_i}\| & -\frac{1}{2}\|W_{xq_i}\|\\
-\frac{1}{2}\|W_{xR_i}\| & \lambda_m[W_{R_i}] & -\frac{1}{2}\|W_{Rq_i}\| \\
-\frac{1}{2}\|W_{xq_i}\| & -\frac{1}{2}\|W_{Rq_i}\| & \lambda_m[W_{q_i}]
\end{bmatrix},\label{eqn:Wi}
\end{align}
where the sub-matrices are given by
\begin{gather*}
W_{x_i} = \frac{1}{n}\begin{bmatrix}
c_x k_{x_0} (1-n\alpha_i\beta) & -\frac{c_x k_{\dot x_0}}{2}(1+n\alpha_i\beta)\\
-\frac{c_x k_{\dot x_0}}{2}(1+n\alpha_i\beta) & k_{\dot x_0}(1-n\alpha_i\beta)-c_x 
\end{bmatrix},\\
W_{R_i} = \frac{1}{n}\begin{bmatrix}
c_R k_{R_0} (1-n\alpha_i\sigma_i) & -\frac{c_R }{2}(k_{\Omega_0}+B+n\alpha_i\sigma_i)\\
-\frac{c_R }{2}(k_{\Omega_0}+B+n\alpha_i\sigma_i) & k_{\Omega_0}(1-n\alpha_i\sigma_i)-2c_R\overline\lambda
\end{bmatrix},\\
W_{q_i}=\begin{bmatrix}
c_q k_q & -\frac{c_qk_\omega}{2}\\
-\frac{c_qk_\omega}{2} & k_\omega-c_q\\
\end{bmatrix},\\
W_{xR_i}= \alpha_i\begin{bmatrix}
\gamma c_x k_{R_0}+\delta_i c_Rk_{x_0} 
& \gamma c_x k_{\Omega_0}+\delta_i k_{x_0}\\
\gamma k_{R_0}+\delta_ic_Rk_{\dot x_0}  &  
\gamma k_{\Omega_0} + \delta_i k_{\dot x_0}
\end{bmatrix},\\
W_{xq_i}= \begin{bmatrix}
c_x B & 0\\
\beta k_{x_0}e_{x_{\max}}+B & 0 
\end{bmatrix},\quad
W_{xR_i}= \begin{bmatrix}
c_R B & 0\\
\alpha_0\sigma_i k_{R_0} +B & 0 
\end{bmatrix}.
\end{gather*}

If the constants $c_x,c_R,c_q$ that are independent of the control input are sufficiently small, the matrices $W_{x_i},W_{R_i},W_{q_i}$ are positive-definite. Also, if the error in the direction of the link is sufficiently small relative to the desired trajectory, we can choose the controller gains such that the matrix $W_i$ is positive-definite, which follows that the zero equilibrium of tracking errors is stable in the sense of Lyapunov, and all of the tracking error variables $z_i$ and the estimation error variables are uniformly bounded, i.e., $e_{x_0},\dot e_{x_0},e_{R_0},e_{\Omega_0},e_{q_i},e_{\omega_i},\tilde\Delta_{x_0},\tilde\Delta_{R_0},\tilde\Delta_{x_i}\in\mathcal{L}_\infty$. These also imply that $e_{x_0},\dot e_{x_0},e_{R_0},e_{\Omega_0},e_{q_i},e_{\omega_i}\in\mathcal{L}_2$ from \refeqn{U0}, and that $\dot e_{x_0},\ddot e_{x_0},\dot e_{R_0},\dot e_{\Omega_0},\dot e_{q_i},\dot e_{\omega_i}\in\mathcal{L}_\infty$. According to Barbalat's lemma~\cite{IoaSun95}, all of the tracking error variables $e_{x_0},\dot e_{x_0},e_{R_0},e_{\Omega_0},e_{q_i},e_{\omega_i}$ and their time-derivatives asymptotically converge to zero. 
\subsection{Proof of Proposition \ref{prop:FDM}}\label{sec:prfFDM}

We first show that the attitude of the $i$-th quadrotor becomes exactly equal to its desired value within a finite time, i.e., $R_i(t)=R_{i_c}(t)$ for any $t\geq T$ for some $T>0$. This is achieved by finite-time stability theory~\cite{BhaBerSJCO00}. This proof is composed of two parts: (i) $s_i(t)=0$ for any $t>T_s$ for some $T_s<\infty$; (ii) when the state is confined to the surface defined by $s_i=0$, we have $e_{R_i}(t)=e_{\Omega_i}(t)=0$ for any $t>T_R$ for some $T_R<\infty$. From now on, we drop the subscript $i$ for simplicity, as the subsequent development is identical for all quadrotors.

From~\cite{LeeLeoPICDC10}, the error dynamics for $e_\Omega$ is given by
\begin{align*}
J\dot e_{\Omega} & = -\Omega\times \Omega +M +\Delta_R+ J(\hat\Omega R^T R_c \Omega_c - R^T R_c\dot\Omega_c).
\end{align*}
Substituting \refeqn{Mi},
\begin{align}
J\dot e_{\Omega} & = -k_s s - l_s S(r,s)+\Delta_R-B_\delta\frac{s}{\|s\|}\nonumber\\
&\quad - (k_R J+ l_s r J\mathrm{diag}_j[|e_{R_{j}}|^{r-1}]) E(R,R_{c})e_\Omega  .\label{eqn:JeWi_dot}
\end{align}

Let a Lyapunov function be
\begin{align*}
\mathcal{W} & = \frac{1}{2} s\cdot Js.
\end{align*}
From \refeqn{si} and \refeqn{eRi_dot}, its time-derivative is given by
\begin{align*}
\dot{\mathcal{W}} & = s\cdot\{J\dot e_\Omega + (k_R J + l_s r J\mathrm{diag}_j[|e_{R_j}|^{r-1}]) E(R,R_c) e_\Omega\}.
\end{align*}
Substituting \refeqn{JeWi_dot} and \refeqn{Mi}, and using \refeqn{Bdelta}, it reduces to
\begin{align*}
\dot{\mathcal{W}} & = s\cdot\{-k_s s -l_s S(r,s)+\Delta_R - \frac{s}{\|s\|}B_\delta\}\\
& \leq -k_s \|s\|^2 - l_s \sum_{j=1}^n |s_j|^{r+1}+B_\delta \|s\| -B_\delta \|s\| \\
& \leq -k_s\|s\|^2 - l_s \|s\|^{r+1},
\end{align*}
where the last inequality is obtained from the fact that $\|x\|^\alpha\leq \sum_{i=1}^n|x_i|^\alpha$ for any $x=[x_1,\ldots,x_n]^T$ and $0<\alpha<2$~\cite[Lemma 2]{WuRadAA11}. Therefore, 
\begin{align*}
\dot{\mathcal{W}} & \leq -\epsilon_1 \mathcal{W}  -\epsilon_2\mathcal{W}^{(r+1)/2},
\end{align*}
where $\epsilon_1=\frac{2k_s}{\lambda_{M}[J]}$ and $\epsilon_2=l_s(\frac{2}{\lambda_{M}[J]})^{(r+1)/2}$. This implies that $s(t)=0$ for any $t \geq T_s$, where the settling time $T_s$ satisfies
\begin{align*}
T_s\leq \frac{2}{\epsilon_1(1-r)}\ln \frac{\epsilon_1 \mathcal{W}(0)^{(1-r)/2}+\epsilon_2}{\epsilon_2},
\end{align*}
according to~\cite[Remark 2]{YuYuA05}.

Next, consider the second part of the proof when $s=0$. Let a configuration error function for the attitude of a quadrotor be
\begin{align*}
\Psi_R =\frac{1}{2}\trs{I-R_c^T R}.
\end{align*}
Consider a domain give by $D_R=\{(R,\Omega)\in\SO\times\Re^3\,|\, \Psi_R < \psi_R < 2\}$. It has been shown that the following inequality is satisfied in the domain,
\begin{align}
\frac{1}{2}\|e_R\|^2 \leq \Psi_R \leq \frac{1}{2-\psi_R} \|e_R\|^2.\label{eqn:PsiB}
\end{align}
Therefore, it is positive-definite about $e_R=0$. The time-derivative of $\Psi_R$ is given by $\dot\Psi_R= e_R\cdot e_\Omega$. Therefore, when $s=0$, we have
\begin{align*}
\dot\Psi_R & = 
-k_{R} \|e_{R}\|^2 - l_{R} \sum_{j=1}^n |e_{R_j}|^{r+1}\\
& \leq 
-k_{R} \|e_{R}\|^2 - l_{R} \|e_R\|^{r+1},
\end{align*}
Substituting \refeqn{PsiB}, we obtain 
\begin{align*}
\dot\Psi_R & \leq -\epsilon_3\Psi_R - \epsilon_4\Psi_R^{(r+1)/2},
\end{align*}
where $\epsilon_3=\frac{k_R}{2-\psi_R}$ and $\epsilon_4=\frac{l_R }{(2-\psi_R)^{(r+1)/2}}$.
This implies that $e_R(t)=e_\Omega(t)=0$ for any $t \geq T_R$, where the settling time $T_R$ satisfies
\begin{align*}
T_R\leq \frac{2}{\epsilon_3(1-r)}\ln \frac{\epsilon_3 \Psi_R(0)^{(1-r)/2}+\epsilon_4}{\epsilon_4}.
\end{align*}

In summary, whenever $t\geq T^*\triangleq \max  \{T_s,T_R\}$, it is guaranteed that $R_i(t)=R_{i_c}(t)$ for the $i$-th quadrotor. Next, we consider the \textit{reduced system}, which corresponds to the dynamics of the payload and the rotational dynamics of the links when $R_i(t)\equiv R_{i_c}(t)$. From \refeqn{fi} and \refeqn{b3i}, the control force of quadrotors when $R_i=R_{i_c}$ is given by
\begin{align*}
-f_i \cdot R_i e_3 = (u_i\cdot R_{c_i} e_3) R_{c_i} e_3 = 
(u_i\cdot -\frac{u_i}{\|u_i\|}) -\frac{u_i}{\|u_i\|} = u_i.
\end{align*}
Therefore, the reduced system is given by the controlled dynamics of the simplified model. 

If the controller gains $k_R,l_R,k_s,l_s$ are selected large such that $T^*$ is sufficiently small, the solution stays inside of the domain $D$, where the stability results of Proposition 1 hold, during $0\leq t<T^*$. After $t\geq T^*$, the controlled system corresponds to the controlled system of the simplified dynamic model, and from Proposition \ref{prop:SDM}, the tracking errors asymptotically coverage to zero, and the estimation error are uniformly bounded.

\bibliography{/Users/tylee/Documents/BibMaster}

\begin{thebibliography}{10}
\providecommand{\url}[1]{#1}
\csname url@rmstyle\endcsname
\providecommand{\newblock}{\relax}
\providecommand{\bibinfo}[2]{#2}
\providecommand\BIBentrySTDinterwordspacing{\spaceskip=0pt\relax}
\providecommand\BIBentryALTinterwordstretchfactor{4}
\providecommand\BIBentryALTinterwordspacing{\spaceskip=\fontdimen2\font plus
\BIBentryALTinterwordstretchfactor\fontdimen3\font minus
  \fontdimen4\font\relax}
\providecommand\BIBforeignlanguage[2]{{%
\expandafter\ifx\csname l@#1\endcsname\relax
\typeout{** WARNING: IEEEtran.bst: No hyphenation pattern has been}%
\typeout{** loaded for the language `#1'. Using the pattern for}%
\typeout{** the default language instead.}%
\else
\language=\csname l@#1\endcsname
\fi
#2}}

\bibitem{PalCruIRAM12}
I.~Palunko, P.~Cruz, and R.~Fierro, ``Agile load transportation,'' \emph{IEEE
  Robotics and Automation Magazine}, vol.~19, no.~3, pp. 69--79, 2012.

\bibitem{MicFinAR11}
N.~Michael, J.~Fink, and V.~Kumar, ``Cooperative manipulation and
  transportation with aerial robots,'' \emph{Autonomous Robots}, vol.~30, pp.
  73--86, 2011.

\bibitem{MazKonJIRS10}
I.~Maza, K.~Kondak, M.~Bernard, and A.~Ollero, ``Multi-{UAV} cooperation and
  control for load transportation and deployment,'' \emph{Journal of
  Intelligent and Robotic Systems}, vol.~57, pp. 417--449, 2010.

\bibitem{LeeSrePICDC13}
T.~Lee, K.~Sreenath, and V.~Kumar, ``Geometric control of cooperating multiple
  quadrotor {UAV}s with a suspended load,'' in \emph{Proceedings of the IEEE
  Conference on Decision and Control}, vol. 5510--5515, Florence, Italy, Dec.
  2013.

\bibitem{GooLeePACC14}
F.~Goodarzi, D.~Lee, and T.~Lee, ``Geometric stabilization of a quadrotor {UAV}
  with a payload connected by flexible cable,'' in \emph{Proceedings of the
  American Control Conference}, June 2014, pp. 4925--4930.

\bibitem{LeePICDC14}
T.~Lee, ``Geometric control of multiple quadrotor {UAVs} transporting a
  cable-suspended rigid body,'' in \emph{Proceedings of the IEEE Conference on
  Decision and Control}, Dec. 2014, accepted.

\bibitem{LeeLeoPICDC10}
T.~Lee, M.~Leok, and N.~McClamroch, ``Geometric tracking control of a quadrotor
  aerial vehicle on {$\SE$},'' in \emph{Proceedings of the IEEE Conference on
  Decision and Control}, Atlanta, GA, Dec. 2010, pp. 5420--5425.

\bibitem{LeeACC151ext}
\BIBentryALTinterwordspacing
T.~Lee, ``Geometric adaptive control of quadrotor {UAVs} transporting a
  cable-suspended rigid body,'' 2014. [Online]. Available:
  \url{http://fdcl.seas.gwu.edu/ACC15.1.ext.pdf}
\BIBentrySTDinterwordspacing

\bibitem{GooLeePECC13}
F.~Goodarzi, D.~Lee, and T.~Lee, ``Geometric nonlinear {PID} control of a
  quadrotor {UAV} on {$\SE$},'' in \emph{Proceedings of the European Control
  Conference}, Zurich, July 2013, pp. 3845--3850.

\bibitem{BulLew05}
F.~Bullo and A.~Lewis, \emph{Geometric control of mechanical systems}, ser.
  Texts in Applied Mathematics.\hskip 1em plus 0.5em minus 0.4em\relax New
  York: Springer-Verlag, 2005, vol.~49, modeling, analysis, and design for
  simple mechanical control systems.

\bibitem{Wu12}
T.~Wu, ``Spacecraft relative attitude formation tracking on {SO(3)} based on
  line-of-sight measurements,'' Master's thesis, The George Washington
  University, 2012.

\bibitem{IoaSun95}
P.~Ioannou and J.~Sung, \emph{Robust Adaptive Control}.\hskip 1em plus 0.5em
  minus 0.4em\relax Prentice Hall, 1995.

\bibitem{Kha96}
H.~Khalil, \emph{Nonlinear Systems}, 2nd Edition, Ed.\hskip 1em plus 0.5em
  minus 0.4em\relax Prentice Hall, 1996.

\bibitem{BhaBerSJCO00}
S.~Bhat and D.~Bernstein, ``Finite-time stability of continuous autonomous
  systems,'' \emph{SIAM Journal of Control and Optimization}, vol.~38, no.~3,
  pp. 751--766, 2000.

\bibitem{YuYuA05}
S.~Yu, X.~Yu, B.~Shirinzadeh, and Z.~Man, ``Continuous finite-time control for
  robotic manipulators with terminal slideing mode,'' \emph{Automatica},
  vol.~41, pp. 1957--1964, 2005.

\bibitem{WuRadAA11}
S.~Wu, G.~Radice, Y.~Gao, and Z.~Sun, ``Quaternion-based finite time control
  for spacecraft attitude tracking,'' \emph{Acta Astronautica}, vol.~69, pp.
  48--58, 2011.

\end{thebibliography}
\bibliographystyle{IEEEtran}

\end{document}